\documentclass[oneside,english]{amsart}
\usepackage[T1]{fontenc}
\usepackage[latin9]{inputenc}
\usepackage{color}
\usepackage{babel}
\usepackage{units}
\usepackage{amstext}
\usepackage{amsthm}
\usepackage{amssymb}
\usepackage[unicode=true,pdfusetitle,
 bookmarks=true,bookmarksnumbered=false,bookmarksopen=false,
 breaklinks=false,pdfborder={0 0 0},pdfborderstyle={},backref=false,colorlinks=true]
 {hyperref}

\makeatletter
\numberwithin{equation}{section}
\numberwithin{figure}{section}
  \theoremstyle{remark}
  \newtheorem*{rem*}{\protect\remarkname}
\theoremstyle{plain}
\newtheorem{thm}{\protect\theoremname}[section]
  \theoremstyle{plain}
  \newtheorem{lem}[thm]{\protect\lemmaname}
  \theoremstyle{definition}
  \newtheorem{defn}[thm]{\protect\definitionname}
  \theoremstyle{plain}
  \newtheorem{prop}[thm]{\protect\propositionname}
  \theoremstyle{plain}
  \newtheorem{cor}[thm]{\protect\corollaryname}
  \theoremstyle{definition}
  \newtheorem{problem}[thm]{\protect\problemname}

\makeatother

  \providecommand{\corollaryname}{Corollary}
  \providecommand{\definitionname}{Definition}
  \providecommand{\lemmaname}{Lemma}
  \providecommand{\problemname}{Problem}
  \providecommand{\propositionname}{Proposition}
  \providecommand{\remarkname}{Remark}
\providecommand{\theoremname}{Theorem}

\begin{document}

\title{Sectional curvature-type conditions on metric spaces}

\author{Martin Kell}

\date{\today}

\address{Mathematisches Institut, Universität Tübingen, Tübingen, Germany}

\email{martin.kell@math.uni-tuebingen.de}
\begin{abstract}
In the first part Busemann concavity as non-negative curvature is
introduced and a bi-Lipschitz splitting theorem is shown. Furthermore,
if the Hausdorff measure of a Busemann concave space is non-trivial
then the space is doubling and satisfies a Poincar\'e condition and
the measure contraction property. Using a comparison geometry variant
for general lower curvature bounds $k\in\mathbb{R}$, a Bonnet-Myers
theorem can be proven for spaces with lower curvature bound $k>0$.

In the second part the notion of uniform smoothness known from the
theory of Banach spaces is applied to metric spaces. It is shown that
Busemann functions are (quasi-)convex. This implies the existence
of a weak soul. In the end properties are developed to further dissect
the soul.
\end{abstract}

\thanks{Part of this research was conducted while the author was an EPDI-PostDoc
at the IH\'ES. He wants to thank the IH\'ES for its hospitality
and financial support. The author is grateful to the referee for carefully
reading through the manuscript and for encouraging to prove a Bonnet-Myers
theorem.}

\maketitle
\global\long\def\diam{\operatorname{diam}}
\global\long\def\interior{\operatorname{int}}
\global\long\def\closure{\operatorname{cl}}
\global\long\def\Con{\operatorname{Con}}
\global\long\def\diam{\operatorname{diam}}

In order to understand the influence of curvature on the geometry
of a space it helps to develop a synthetic notion. Via comparison
geometry sectional curvature bounds can be obtained by demanding that
triangles are thinner or fatter than the corresponding comparison
triangles. The two classes are called $CAT(\kappa)$- and resp. $CBB(\kappa)$-spaces.
We refer to the book \cite{Bridson1999} and the forthcoming book
\cite{Alexander} (see also \cite{Burago1992,Otsu1997}). Note that
all those notions imply a Riemannian character of the metric space.
In particular, the angle between two geodesics starting at a common
point is well-defined. Busemann investigated a weaker notion of non-positive
curvature which also applies to normed spaces \cite[Section 36]{Busemann1955}.
A similar idea was developed by Pedersen \cite{Pedersen1952} (see
also \cite[(36.15)]{Busemann1955}). Pedersen's conditions is better
suited for the study of Hilbert geometries, see \cite{Kelly1958}.
In \cite{Ohta2007} Ohta studied even weaker convexity notion, called
$L$-convexity which can be seen as a relaxed form of Busemann's non-positive
curvature assumption.

In the recent year a synthetic notion of a lower bound on the Ricci
curvature was defined by Lott-Villani \cite{LV2009} and Sturm \cite{Sturm2006a}.
Surprisingly, their condition include also Finsler manifolds \cite{Ohta2009,Ohta2013}.
The notion of lower curvature bound in the sense of Alexandrov, i.e.
$CBB(\kappa)$-spaces, is compatible with this Ricci bound \cite{Petrunin2010,Gigli2013}.
However, by now there is no known sectional curvature analogue for
Finsler manifolds which is compatible to Ohta's Ricci curvature bounds
and thus the synthetic Ricci bounds. 

In this note we present two approaches towards a sectional curvature-type
condition. The first is the ``converse'' of Busemann's non-positive
curvature condition. This condition implies a bi-Lipschitz splitting
theorem, uniqueness of tangent cones and if the space admits a non-trivial
Hausdorff measure then it satisfies doubling and Poincar\'e conditions,
and even the measure contraction property. This approach rather focuses
on the generalized angles formed by two geodesics. Using ideas from
comparison geometry one can easily define general lower curvature
bounds and prove a Bonnet-Myers theorem if the lower bound is positive.

The second approach can be seen as a dual to the theory of uniformly
convex metric spaces which were studied in \cite{Ohta2007,Kuwae2013,Kell2014}.
We call this condition uniform smoothness. This rather weak condition
is only powerful in the large as any compact Finsler manifold is $2$-uniformly
smooth, see \cite[Corollary 4.4]{Ohta2008}. Nevertheless, if the
spaces is unbounded then Busemann functions associated to rays are
(quasi-)convex and the space has a weak soul. In order to match the
theory in the smooth setting we try to develop further assumptions
which imply existence of a retractions onto the soul and a more local
curvature assumption in terms of Gromov's characterization of non-negative
curvature \cite{Gromov1991}. 

\section{Preliminaries}

Throughout this manuscript let $(X,d)$ be a \emph{proper geodesic
metric space}, i.e. $(X,d)$ is a complete metric space such that
every bounded closed subset is compact and for every $x,y\in X$ there
is a continuous map $\gamma:[0,1]\to X$ with $\gamma(0)=x$, $\gamma(1)=y$
and 
\[
d(\gamma(t),\gamma(s))=|t-s|d(x,y).
\]
The map $\gamma$ will be called a (constant speed) \emph{geodesic}
connecting $x$ and $y$ which is parametrized by $[0,1]$. We say
that a continuous curve $\gamma:I\to X$ defined on some interval
$I\subset\mathbb{R}$ is a \emph{locally geodesic} if for all $t\in I$
there is an interval $I_{t}\subset I$ with $t\in\interior I_{t}$
such that $\gamma$ restricted to $I_{t}$ is a constant speed geodesic.

We say that $(X,d)$ is \emph{non-branching} if for all geodesics
$\gamma,\eta:[0,1]\to X$ with $\gamma(0)=\eta(0)$ and $\gamma(t)=\eta(t)$
for some $t\in(0,1)$ it holds $\gamma(t)=\eta(t)$ for all $t\in[0,1]$.
In other words two geodesic start at the same point and intersect
in the middle must agree.

\subsection{Notions of convexity }

A function $f:X\to\mathbb{R}$ is said to be (geodesically) \emph{convex}
if $t\mapsto f(\gamma(t))$ is convex for all geodesics $\gamma:[0,1]\to X$,
i.e. 
\[
f(\gamma(t))\le(1-t)f(\gamma(0))+tf(\gamma(1)).
\]
We say that $f$ is \emph{$p$-convex} if $f^{p}$ is convex. Furthermore,
$f$ is \emph{concave} if $-f$ is convex, and if $f$ is both convex
and concave then it is said to be \emph{affine}. The limiting notion
$p\to\infty$ is usually called quasi-convex. More precisely, we say
that a function $f:X\to\mathbb{R}$ is said to be \emph{quasi-convex}
if 
\[
f(\gamma(t))\le\max\{f(\gamma(0)),f(\gamma(1))\}.
\]
It is said to be \emph{strictly quasi-convex} if the inequality is
strict whenever $\gamma(0)\ne\gamma(1)$. Furthermore, we say $f$
is \emph{properly quasi-convex} if the inequality is strict whenever
$f$ restricted to $\gamma$ is non-constant. Note that any $p$-convex
function is automatically properly quasi-convex. As above \emph{quasi-concavity}
of $f$ is just quasi-convexity of $-f$, and $f$ is said to be \emph{monotone}
if it both quasi-convex and quasi-concave. 

A nice construction is obtained as follows: Let $h:\mathbb{R}\to\mathbb{R}$
be a non-decreasing function. Then $h\circ f$ is quasi-convex for
any quasi-convex function $f$. Furthermore, if $h$ is strictly increasing
then $h\circ f$ is strictly quasi-convex if $f$ is strictly quasi-convex. 
\begin{rem*}
In \cite{Busemann1983} it is suggested to use the terminology \emph{peaklessness}
for proper quasi-convexity and \emph{weak peaklessness} for quasi-convexity.
However, this terminology does not seem to be frequently used in the
literature. Because we only obtain quasi-convex functions we stay
with the better known term of quasi-convexity.
\end{rem*}
A subset $C$ of $X$ is said to be \emph{convex} if for all geodesics
$\gamma:[0,1]\to X$ with $\gamma(0),\gamma(1)\in C$ it holds $\gamma(t)\in C$
for $t\in(0,1)$. If, in addition, $\gamma(t)\in\interior C$ whenever
$\gamma$ is non-constant and $t\in(0,1)$ then $C$ is said to be
\emph{strictly convex}. A stronger notion, called \emph{totally geodesic},
is obtained by requiring that $C$ also contains all local geodesics,
i.e. if $\gamma:[a,b]\to X$ is locally geodesic with $\gamma(a),\gamma(b)\in C$
then $\gamma(t)\in C$ for $t\in(a,b)$.

As is well-known a function $f:X\to\mathbb{R}$ is (strictly) convex
if the epigraph $\{(x,t)\in X\times\mathbb{R}\,|\,t\ge f(x)\}$ is
(strictly) convex in the product space $(X\times\mathbb{R},\tilde{d})$
where $\tilde{d}((x,t),(y,s))=\sqrt{d(x,y)^{2}+|t-s|^{2}}$. In a
similar way $f$ is quasi-convex if the sublevel $C_{s}=f^{-1}((-\infty,s])$
are totally geodesic. Furthermore, if $f$ is strictly convex then
each $C_{s}$ is strictly convex.

In general strict convexity of the sublevels of $f$ is not related
to strict quasi-convexity of $f$. However, one can always construct
a quasi-convex function out of an exhaustive non-decreasing family
of closed convex sets as follows. Under some additional assumptions
on the family the function is also strict quasi-convex if the sublevels
are strictly convex.
\begin{lem}
Assume $(C_{s})_{s\in I}$, $I\subset\mathbb{R}$, is a non-decreasing
family of closed convex sets with $\cup_{s\in I}C_{s}=X$. Then the
function $f:X\to\mathbb{R}$ defined by
\[
f(x)=\inf\{\arctan s\in\mathbb{R}\,|\,x\in C_{s}\}
\]
is quasi-convex. If, in addition, $(1)$ for each $x\in X$ there
is an $s\in I$ such that $x\notin\interior C_{s}$, and 
\[
(2)\qquad\bigcap_{s'>s}C_{s'}=C_{s}
\]
and 
\[
(3)\quad\bigcup_{s'<s}\interior C_{s'}=\interior C_{s}
\]
then $f$ is strictly quasi-convex iff each $C_{s}$ is strictly convex.
\end{lem}
\begin{rem*}
If $I$ is closed one can alternatively require that 
\end{rem*}
\begin{proof}
The first and second part directly follow from the definition. So
assume $(C_{s})_{s\in I}$ satisfies the additional properties and
let $x,y\in X$ with $f(x)\le f(y)$. The first assumption shows $f(y),f(x)>-\frac{\pi}{2}$
so that $x,y\in C_{\tan f(y)}$ by the second assumption. 

If $f$ is strictly quasi-convex then each sublevel is obviously strictly
convex. We show the converse: Assume each $C_{s}$ is strictly convex.
Then any midpoint $m$ of $x$ and $y$ is in the interior of $C_{\tan f(y)}$.
The second assumption implies that there is an $s<\tan f(y)$ such
that $m\in\interior C_{s'}$. Thus $f(m)\le s<f(y)$ proving strict
quasi-convexity of $f$. 
\end{proof}
Note that all properties above are necessary to show that $f$ is
strictly quasi-convex. Indeed, the family $(C_{s})_{s\in\mathbb{R}}$
of strictly convex closed intervals in $\mathbb{R}$ given by 
\[
C_{s}^{1}=\begin{cases}
[-s,s] & s\le1\\{}
[-(s+1),s+1] & s>1
\end{cases}
\]
satisfies the first and second but not the third property. Similarly,
\[
C_{s}^{2}=\begin{cases}
[-s,s] & s<1\\{}
[-(s+1),s+1] & s\ge1
\end{cases}
\]
satisfies the first and third but not second property. Finally, let
$\phi:\mathbb{R}\to(0,\infty)$ be an increasing homeomorphism. Then
\[
C_{s}^{3}=(-\infty,\phi(s)]
\]
satisfies all but the first property. In either case the induced functions
$f_{i}$, $i=1,2,3$, are constant on some open interval. In particular,
they are not strictly quasi-convex. One may replace the intervals
by geodesic balls resp. horoballs of the same radius if the metric
space has strictly convex geodesic balls and resp. strictly convex
horoballs to get more general examples. 

\subsection{Busemann functions}

A central tool to study the structure of spaces with certain generalized
curvature bounds are Busemann function associated to rays. Here a
ray $\gamma:[0,\infty)\to X$ is an isometric embedding of the half
line, i.e. 
\[
d(\gamma(t),\gamma(s))=|t-s|\qquad s,t\ge0.
\]

\begin{defn}
[Busemann function] Given a ray $\gamma:[0,\infty)\to X$ we define
the \emph{Busemann function} $b_{\gamma}:X\to\mathbb{R}$ as follows
\[
b_{\gamma}(x)=\lim_{t\to\infty}t-d(\gamma(t),x).
\]
\end{defn}
Note that the right hand side is non-decreasing in $t$ so that the
limit is well-defined.

In case $\gamma:\mathbb{R}\to X$ is a line we define $\gamma^{\pm}:[0,\infty)\to X$
by $\gamma^{\pm}(t)=\gamma(\pm t)$. One can show that 
\[
b_{\gamma^{+}}+b_{\gamma^{-}}\le0.
\]

We say a ray $\eta:[0,\infty)\to X$ is \emph{asymptotic to $\gamma$}
if there is sequence $t_{n}\to\infty$ and unit speed geodesics $\eta_{n}:[0,d(\eta(0),\gamma(t_{n}))]\to X$
from $\eta(0)$ to $\gamma(t_{n})$ such that $\eta_{n}$ converges
uniformly on compact subsets to $\eta$. It is not difficult to see
that 
\[
b_{\gamma}(\eta(t))=t+b_{\gamma}(\eta(0))\quad t\ge0.
\]
If $(X,d)$ is a proper then for any $x\in X$ we can select a subsequence
of geodesics $(\eta_{n})_{n\in\mathbb{N}}$ connecting $x$ and $\gamma(t_{n})$
such that $(\eta_{n_{k}})_{k\in\mathbb{N}}$ converges to a ray $\eta$
which is asymptotic to $\gamma$ starting at $x$. 

A line $\eta:\mathbb{R}\to X$ is said to be \emph{bi-asymptotic}
to the line $\gamma:\mathbb{R}\to X$ if $\eta^{\pm}$ is asymptotic
to $\gamma^{\pm}$. We say that $\eta$ is \emph{parallel} to $\gamma$
if the shifted lines $\eta_{(s)}:t\mapsto\eta(t+s)$ are bi-asymptotic
to $\gamma$. It is not clear whether every bi-asymptotic line $\eta$
to $\gamma$ is also parallel to $\gamma$. Note, that if $(X,d)$
is non-branching then it suffices to show that $b_{\gamma^{+}}$ restricted
to $\eta$ is affine. Assuming Busemann concavity this is indeed the
case, see Lemma \ref{lem:bi-asym-parallel} below.

\subsection{Gromov-Hausdorff convergence}

Given two subsets $A$ and $B$ of a metric space $(Z,d_{Z})$ the
\emph{Hausdorff distance} $d_{Z}^{(H)}$ of $A$ and $B$ is defined
as 
\[
d_{Z}^{(H)}(A,B)=\inf\{\epsilon>0\,|\,A\subset B_{\epsilon},B\subset A_{\epsilon}\}
\]
where $A_{\epsilon}=\cup_{x\in A}B_{\epsilon}(x)$ and $B_{\epsilon}=\cup_{x\in B}B_{\epsilon}(x)$.
Let $(X,d_{X})$ and $(Y,d_{Y})$ be two metric spaces. We say that
a metric space $(Z,d_{Z})$ together with two maps $i_{X}:X\to Z$
and $i_{Y}:Y\to Z$ is a \emph{metric coupling} of $(X,d_{X})$ and
$(Y,d_{Y})$ if $i_{X}$ and $i_{Y}$ are isometric embeddings, i.e.
for all $x_{i}\in X$ and $y_{i}\in Y$ it holds 
\begin{eqnarray*}
d_{Z}(i_{X}(x_{1}),i_{X}(x_{2})) & = & d_{X}(x_{1},x_{2})\\
d_{Z}(i_{Y}(y_{1}),i_{X}(y_{2})) & = & d_{Y}(y_{1},y_{2}).
\end{eqnarray*}
 Then the \emph{Gromov-Hausdorff distance} of $(X,d_{X})$ and $(Y,d_{Y})$
is defined as 
\[
d_{GH}((X,d_{X}),(Y,d_{Y}))=\inf d_{Z}^{(H)}(i_{X}(X),i_{Y}(Y))
\]
where the infimum is taken over all metric couplings of $(X,d_{X})$
and $(Y,d_{Y})$. Note that $d_{GH}$ is zero iff the completions
of $(X,d_{X})$ and $(Y,d_{Y})$ are isometric. Thus $d_{GH}$ induces
a metric on the equivalence classes of isometric complete metric spaces.
One may restrict the metric couplings further if a certain point is
supposed to be preserved. More precisely, let $(X,d_{X},x)$ and $(Y,d_{Y},y)$
be \emph{point metric spaces}. Then a metric coupling is a triple
$((Z,d_{Z},z),i_{X},i_{Y})$ such that $i_{X}$ and $i_{Y}$ are isometric
embeddings with $i_{X}(x)=i_{Y}(y)=z$. The Gromov-Hausdorff distance
is then defined as above.

In general this convergence is rather strong in case of non-compact/unbounded
spaces. A weaker notion is given by the \emph{pointed Gromov-Hausdorff
convergence}. More precisely, we say that a sequence $(X_{n},d_{n},x_{n})$
converges to the pointed metric space $(X,d,x)$ in the \emph{pointed
Gromov-Hausdorff topology} if for each $r>0$ 
\[
d_{GH}((B_{r}^{n}(x_{n}),d_{n}),(B_{r}(x),d))\to0
\]
where $B_{r}^{n}(x_{n})$ and $B_{r}(x)$ are the usual balls of radius
$r$ with respect to $d_{n}$ and resp. $d$.

\section{Busemann concavity}

In this section we define a form of non-negative curvature which is
similar to Busemann's notion of non-positive curvature. As it turns
out this notion is not new. It appeared already in the study of Hilbert
geometry as ``has defined curvature'' \cite{Kelly1958} and in a
paper of Kann \cite{Kann1961} who studied two dimensional $G$-spaces
of positive curvature which is defined via an additional quadratic
term. This, however, differs slightly from the definition in terms
of comparison geometry presented below. 
\begin{defn}
[Busemann concave] A geodesic metric space $(X,d)$ is said to be
\emph{Busemann concave} if for any three point $x,y_{1},y_{2}\in X$
and any geodesics $\gamma_{x,y_{i}}$ connecting $x$ and $y_{i}$
the function 
\[
t\mapsto d(\gamma_{x,y_{1}}(t),\gamma_{x,y_{2}}(t))
\]
 is concave on $[0,1]$.
\end{defn}
Busemann concavity implies that the space is non-branching. One readily
verifies that any strictly convex Banach space is Busemann concave.
Below we give further examples.

It is possible to define Busemann concavity in terms of comparison
geometry. More precisely, let $\triangle(\tilde{x}\tilde{y}_{1}\tilde{y}_{2})$
be a comparison triangle in $\mathbb{R}^{2}$ with side lengths $d(x,y_{1})$,
$d(x,y_{2})$ and $d(y_{1},y_{2})$. Then Busemann convexity is equivalent
to requiring 
\[
d(\gamma_{x,y_{1}}(t),\gamma_{x,y_{2}}(t))\ge d_{\mathbb{R}^{2}}(\tilde{\gamma}_{\tilde{x},\tilde{y}_{1}}(t),\tilde{\gamma}_{\tilde{x},\tilde{y}_{2}}(t))
\]
for all $t\in[0,1]$. With the help of this, it is possible to define
spaces with lower bound $k$ on the curvature for general $k\in\mathbb{R}$.
Note that for $k>0$ the existence of a comparison triangle is implicitly
assumed, see also Section \ref{subsec:Bonnet-Myers-theorem}.

Using the triangle comparison definition for Alexandrov spaces and
the Topogonov comparison theorem for Riemannian manifolds we obtain
the following.
\begin{lem}
Every Alexandrov space with sectional curvature bounded below by $k$
has Busemann curvature bounded below by $k$. Furthermore, a Riemannian
manifold has sectional curvature bounded below $k$ iff it has Busemann
curvature bounded below by $k$.
\end{lem}
\begin{rem*}
Similar to the argument in \cite{Foertsch2010} the existence of angles
implies that a metric space with Busemann curvature bounded below
by $k\in\mathbb{R}$ is an Alexandrov space with the same lower curvature
bound. We leave it to the reader to work out the details.
\end{rem*}
In contrast to Busemann convexity it is not clear whether it suffices
to check the property above only for midpoints, i.e. 
\[
d(m_{x,y_{1}},m_{x,y_{2}})\ge\frac{1}{2}d(y_{1},y_{2}).
\]

Note that similar to Busemann convexity, Busemann concavity is not
stable under Gromov-Hausdorff convergence. Nevertheless, a weaker
property is preserved. 
\begin{defn}
A function $\sigma:X\times X\times[0,1]$ is called a \emph{geodesic
bicombing} if $\sigma_{xy}(0)=x$, $\sigma_{xy}(1)=y$ and $d(\sigma_{xy}(t),\sigma_{xy}(t'))=|t-t'|d(x,y)$
for all $x,y\in X$ and $t,t'\in[0,1]$. We say the bicombing is \emph{closed}
if $(x_{n},y_{n})\to(x,y)$ implies $\sigma_{x_{n}y_{n}}(t)\to\sigma_{xy}(t)$
for all $t\in[0,1]$. 
\end{defn}

\begin{defn}
[weak Busemann concavity] A metric space $(X,d)$ is said to be \emph{weak
Busemann concave} if there is a closed geodesic bicombing $\sigma$
such that for all $x,y,z\in X$ it holds 
\[
t\mapsto d(\sigma_{xy}(t),\sigma_{xz}(t))
\]
is concave.
\end{defn}
\begin{rem*}
This property resembles Kleiner's notion of often convex spaces \cite{Kleiner1997},
resp. the notion of convex bicombings \cite{Descombes2014}. 
\end{rem*}
It is easy to see that any Banach space is weakly Busemann concave.
The corresponding geodesic bicombing is given by straight lines. In
a future work we try to give generalizations of Theorem \ref{thm:Splitting}
and Proposition \ref{prop:Carnot-Busemann} using only weak Busemann
concavity.

Below it is shown that tangent cones of Busemann concave spaces are
uniquely defined. If the space is doubling or admits a doubling measure
then one can adjust the proofs of \cite{LeDonne2011} to show that
the tangent cones are (locally compact) Carnot groups away from a
thin set, i.e. a set which has zero measure for every doubling measure.
We refer to \cite{LeDonne2011} for necessary definitions of Carnot
groups. 

The following shows that the only Busemann concave Carnot groups are
Banach spaces with strictly convex norm. As Busemann concavity is
not stable under Gromov-Hausdorff convergence this is not sufficient
to conclude that almost all tangent cones are Banach spaces.
\begin{prop}
\label{prop:Carnot-Busemann}Assume $(X,d)$ is a finite dimensional
Busemann concave Carnot group with Carnot-Caratheodory metric $d$.
Then $(X,d)$ is a Banach space with strictly convex norm.
\end{prop}
\begin{proof}
Let $V_{1}\in TX$ be a horizontal vector and $\eta^{1},\eta^{2}:(-\epsilon,\epsilon)\to X$
be tangent to $V_{1}$ at $\eta^{1}(0)=\eta^{2}(0)$. Now define the
delated curves 
\[
\eta_{\lambda}^{i}(t)=\delta_{\lambda}(\eta^{i}(t)).
\]
Then by invariance it holds $d(\eta_{\lambda}^{1}(t),\eta_{\lambda}^{2}(t))=\lambda d(\eta^{1}(t),\eta^{2}(t)).$
Together with Busemann concavity we get for $\lambda>1$ 
\[
d\left(\eta_{\lambda}^{1}\left(\frac{t}{\lambda}\right),\eta_{\lambda}^{2}\left(\frac{t}{\lambda}\right)\right)\ge d(\eta^{1}(t),\eta^{2}(t)).
\]
 As $\eta^{1}$ and $\eta^{2}$ are both tangent to $V_{1}$, their
Pansu differential is also $V_{1}$. In particular, it holds 
\[
\eta_{\lambda}^{i}\left(\frac{t}{\lambda}\right)\to\exp(tV_{1}).
\]
 But this implies that $d\left(\eta_{\lambda}^{1}\left(\frac{t}{\lambda}\right),\eta_{\lambda}^{2}\left(\frac{t}{\lambda}\right)\right)\to0$
and thus $d(\eta^{1}(t),\eta^{2}(t))=0$. 

To conclude, we just need to note that any $k$-step Carnot group
with $k>1$ has distinct geodesics which are tangent to the same horizontal
vector. Thus $X$ must be a rank $1$ Carnot group, i.e. a Banach
space. Strict convexity of the norm follows as otherwise Busemann
concavity cannot hold.
\end{proof}
\begin{rem*}
[Asymmetric metrics] In principle it is possible to define Busemann
concavity also for asymmetric metrics. In order to prove Propositions
\ref{prop:Busemann-contraction} and \ref{prop:Busemann-MCP-et-al}
one only needs the concavity property between any two forward resp.
backward geodesics starting at a fixed point $x\in X$. The proof
of the splitting theorem then only requires minor adjustments taking
care of the asymmetry of the distance.
\end{rem*}

\subsection{Constructions and examples\label{subsec:Constructions-and-examples}}

A whole family of Busemann concave spaces is obtained by products
of Busemann concave spaces.
\begin{lem}
\label{lem:products}If $(X_{i},d_{i})_{i\in I}$ are Busemann concave
spaces for some finite index set $I\subset\mathbb{N}$ and $F$ is
a strictly convex norm on $\mathbb{R}^{|I|}$ then $X=\times_{i\in I}X_{i}$
equipped with the metric $d((x_{i})_{i\in I},(y_{i})_{i\in I})=F((d(x_{i},y_{i}))_{i\in I})$
is Busemann concave. 
\end{lem}
\begin{rem*}
(1) It is possible to allow for countably infinite index sets. The
obtained space is then an extended metric space, i.e. the metric may
be infinite. 

(2) If $F$ is not strictly convex or some factors are only weakly
Busemann concave then their product is weakly Busemann concave. 
\end{rem*}
\begin{proof}
If $F$ is strictly convex then geodesics in $X$ are obtained as
``product geodesics''. Thus
\begin{eqnarray*}
d((\gamma_{i}(t))_{i\in I},(\gamma_{i}^{'}(t))_{i\in I}) & = & F((d(\gamma_{i}(t),\gamma_{i}^{'}(t))_{i\in I})\\
 & \ge & tF(d(\gamma_{i}(1),\gamma_{i}^{'}(1))_{i\in I})\\
 & = & td((\gamma_{i}(1))_{i\in I},(\gamma_{i}^{'}(1))_{i\in I}).
\end{eqnarray*}
\end{proof}
\begin{rem*}
[Berwald spaces] Using Jacobi fields it is possible to show that
Berwald spaces with flag curvature bounded below by $k\in\mathbb{R}$
satisfy the corresponding Busemann comparison result locally until
the conjugate radius is reached. In \cite{Kristaly2006} it was shown
that in the class Berwald spaces non-positive flag-curvature is equivalent
to Busemann convexity. In the current setting we were not able to
``invert'' the inequalities to show the same for non-negative flag
curvature and obtain global Busemann concavity. Note, however, that
any simply connected non-negatively curved Berwald space which does
not contain a higher rank symmetric factor can be exactly described
\cite{Szabo1981}. More precisely, they are metric products as above
and each factor is either flat or a Riemannian manifold of non-negative
curvature (see \cite{Kell2015}). Thus a deeper understanding of higher
rank symmetric spaces and their Riemannian and Berwald structures
would allow to characterize non-negatively curved Berwald spaces in
terms of Busemann concavity.
\end{rem*}
Recall that the (Euclidean) cone $\Con(X)$ of a metric space $(X,d)$
is the set $X\times[0,\infty)$ where the points $(x,0)$ are identified
and the metric on $\Con(X)$ is given by 
\[
d((x,r),(y,s))^{2}=r^{2}+s^{2}-2rs\cos(\min\{\pi,d(x,y)\}).
\]
 
\begin{lem}
Assume $(X,d)$ has its diameter is bounded by $\pi$. Then the Euclidean
cone over $\Con(X)$ is Busemann concave iff $(X,d)$ has Busemann
curvature bounded below by $1$.
\end{lem}
\begin{rem*}
(1) The bound $\pi$ ensures that there is a comparison triangle.
By Theorem \ref{thm:Bonnet-Myers} this bound always holds if the
space is not $1$-dimensional.

(2) More generally, one can define general $k$-cones $\Con_{k}(X)$,
called spherical suspension for $k=1$ and elliptical cone for $k=-1$,
see \cite[Section 4.3]{Burago1992} and \cite[Definition 5.6]{Bridson1999}.
A similar proof would show that they have Busemann curvature bounded
below by $k$ iff $X$ has Busemann curvature bounded below by $1$
and diameter at most $\pi$. 

(3) Doing the proof backwards implies that $\Con_{k}(X)$ has Busemann
curvature bounded below by $k$ iff $X$ has Busemann curvature bounded
below by $1$. 
\end{rem*}
\begin{proof}
Let $x_{1},x_{2}$ and $x_{3}$ be points in $X$ and $m_{1}$ and
$m_{2}$ the $t$-midpoints of $(x_{1},x_{3})$ and resp. $(x_{2},x_{3})$.
The curvature bound of $X$ translates to 
\[
d(m_{1},m_{2})\ge d_{S^{2}}(\tilde{m}_{1},\tilde{m}_{2})
\]
where $\tilde{m}_{1}$ and $\tilde{m}_{2}$ are the $t$-midpoints
in the comparison triangle whose existence is ensured by the bound
on the diameter. But then $\cos(d(m_{1},m_{2}))\le\cos(d_{S^{2}}(\tilde{m}_{1},\tilde{m}_{2}))$.

Now let $(x_{i},r_{i})$ be three point in $\Con(X)$ and $(m_{1},s_{1})$
and $(m_{2},s_{2})$ be the corresponding $t$-midpoints. With the
help of the comparison space we obtain
\begin{eqnarray*}
t^{2}d((x_{1},r_{1}),(x_{2},r_{2}))^{2} & = & t^{2}d_{\mathbb{R}^{3}}((\tilde{x}_{1},r_{1}),(\tilde{x}_{2},r_{2}))^{2}\\
 & = & d_{\mathbb{R}^{3}}((\tilde{m}_{1},s_{1}),(\tilde{m}_{2},s_{2}))^{2}\\
 & = & s_{1}^{2}+s_{2}^{2}-2s_{1}s_{2}\cos(d_{S^{2}}(\tilde{m}_{1},\tilde{m}_{2}))\\
 & \le & s_{1}^{2}+s_{2}^{2}-2s_{1}s_{2}\cos(d(m_{1},m_{2}))\\
 & = & d((m_{1},s_{1}),(m_{2},s_{2}))^{2}.
\end{eqnarray*}

Assuming conversely that $\Con(X)$ is Busemann concave we see that
$\cos(d(m_{1},m_{2}))\le\cos(d(\tilde{m}_{1},\tilde{m}_{2}))$. Since
$\Con(X)$ is non-branching also $\diam X\le\pi$. Thus $d(\tilde{m}_{1},\tilde{m}_{2})\le\frac{\pi}{2}$
which implies by monotonicity of cosine on $[0,\frac{\pi}{2}]$ the
required comparison
\[
d(m_{1},m_{2})\ge d_{S^{2}}(\tilde{m}_{1},\tilde{m}_{2}).
\]
\end{proof}
An open problem is whether quotients via isometry actions of Busemann
concave spaces are still Busemann concave. The current proofs in the
Alexandrov setting rely heavily on the notion of angle which is not
present in the current setting. This would also imply that any non-negatively
curved Berwald space whose connection does not have a higher rank
symmetric factor is Busemann concave.  

\subsection{Busemann functions, lines and a splitting theorem\label{subsec:splitting}}

For non-negatively curved Riemannian manifolds the existence of a
line implies that the space splits, i.e. there is a metric spaces
$X'$ such that $X$ is isometric/diffeomorphic to $X'\times\mathbb{R}$.
A key point of the splitting theorem is the existence of a unique
line $\eta_{x}$ through every $x\in X$ which is parallel to a given
line $\gamma$. This is usually done by showing that the Busemann
functions $b_{\gamma^{\pm}}$ are affine and the rays asymptotic to
$\gamma^{\pm}$ can be glued to lines. In this section we show more
directly that the gluing property holds and that the space also splits
into a product. However, it is not clear whether the Busemann functions
associated to lines are affine or whether their level sets are convex. 

First the following useful lemma.
\begin{lem}
\label{lem:fraction-busemann}Let $\gamma^{+}:[0,\infty)\to X$ be
a ray then it holds 
\[
\lim_{t\to\infty}\frac{d(x,\gamma(t))}{t}=1.
\]
\end{lem}
\begin{proof}
By definition $b_{\gamma}(x)$ is the limit of the non-decreasing
bounded sequence $t-d(x,\gamma(t))$. Thus 
\[
0=\lim_{t\to\infty}\frac{t-d(x,\gamma(t))}{t}=1-\lim_{t\to\infty}\frac{d(x,\gamma(t))}{t}.
\]
\end{proof}
The following proposition shows that moving in the ray direction induces
a natural expansion. In case of a line one may also move in the opposite
direction to show that this movement is an isometry, see Lemma \ref{lem:moving-isometry}
below.
\begin{prop}
\label{prop:Busemann-contraction}Let $\gamma:[0,\infty)\to X$ be
a ray, and $\eta$ and $\xi$ be rays asymptotic to $\gamma$ that
are generated by the same sequence $t_{n}\to\infty$. Then it holds
\[
d(\eta(t),\xi(s))\le d(\eta(t+a),\xi(s+a))
\]
for all $t,s,a\ge0$.
\end{prop}
\begin{proof}
From the assumption there are sequences $(\eta_{n})_{n\in\mathbb{N}}$
and $(\xi_{n})_{n\in\mathbb{N}}$ where $\eta_{n}$ and resp. $\xi_{n}$
are geodesics from $\eta(0)$ to $\gamma(t_{n})$ and resp. $\xi(0)$
to $\gamma(t_{n})$ such that for any $t\ge0$ it holds $\eta_{n}(t)\to\eta(t)$
and $\xi_{n}(t)\to\xi(t)$. Now fix $s,t,a\ge0$ and set $b_{n}=d(\eta(0),\gamma(t_{n}))$
and $c_{n}=d(\xi(0),\gamma(t_{n}))$. Define geodesics $\bar{\eta}_{n}(r)=\eta_{n}(b_{n}-r)$
and $\bar{\xi}_{n}(r)=\xi_{n}(c_{n}-r)$. Also define $\bar{b}_{n}=b_{n}-t$
and $\bar{c}_{n}=c_{n}-s$ and note that $\bar{b}_{n},\bar{c}_{n}>0$
for sufficiently large $n$. 

By Busemann concavity applied the hinge formed by $\bar{\eta}_{n}$
and $\bar{\xi}_{n}$ with contraction factor $\lambda_{n}=1-\nicefrac{a}{\bar{b}_{n}}$
we have
\begin{eqnarray*}
d(\bar{\eta}_{n}(\lambda_{n}\bar{b}_{n}),\bar{\xi}_{n}(\lambda_{n}\bar{c}_{n})) & \ge & \lambda_{n}d(\bar{\eta}_{n}(\bar{b}_{n}),\bar{\xi}_{n}(\bar{c}_{n}))\\
 & = & \lambda_{n}d(\eta_{n}(t),\xi_{n}(s)).
\end{eqnarray*}

Note that 
\begin{eqnarray*}
\bar{\eta}_{n}(\lambda_{n}\bar{b}_{n}) & = & \eta_{n}(t+a)\\
\bar{\xi}_{n}(\lambda_{n}\bar{c}_{n}) & = & \xi_{n}\left(s+\frac{\bar{c}_{n}}{\bar{b}_{n}}a\right).
\end{eqnarray*}
By Lemma \ref{lem:fraction-busemann} we have $\frac{\bar{c}_{n}}{\bar{b}_{n}}\to1$
so that $\bar{\xi}_{n}(\lambda_{n}\bar{c}_{n})\to\xi(s+a)$ and hence
\[
d(\eta(t),\xi(s))\le d(\eta(t+a),\xi(s+a)).
\]
\end{proof}
We will now prove the splitting theorem in a sequence of lemmas. Assume
in the following that $(X,d)$ is Busemann concave and contains a
line $\gamma:\mathbb{R}\to X$. Denote by $\eta_{x}^{\pm}$ the rays
asymptotic to $\gamma^{\pm}$ and let 
\[
\eta_{x}(t)=\begin{cases}
\eta_{x}^{+}(t) & t\ge0\\
\eta_{x}^{-}(-t) & t\le0.
\end{cases}
\]

\begin{lem}
\label{lem:bi-asym-parallel}The Busemann functions $b_{\gamma^{\pm}}$
are affine when restricted to $\eta_{x}$. Furthermore, $\eta_{x}$
is a bi-asymptotic to $\gamma$.
\end{lem}
\begin{proof}
We only need to prove that 
\[
b_{\gamma^{+}}(\eta_{x}^{-}(s))=b_{\gamma^{+}}(x)-s
\]
for $s\ge0$. Indeed, this will show that $b_{\gamma^{+}}$ is affine
on $\eta_{x}$. A similar argument also works for $b_{\gamma^{-}}$.

From the previous lemma we have 
\[
d(\eta_{x}^{-}(s),\gamma(t))\ge d(\eta_{x}^{-}(0),\gamma(t+s)).
\]
 Thus 
\[
t-d(\eta_{x}^{-}(s),\gamma(t))\le t+s-d(x,\gamma(t+s))-s.
\]
Taking the limit as $t\to\infty$ we obtain 
\[
b_{\gamma^{+}}(\eta_{x}^{-}(s))\le b_{\gamma^{+}}(x)-s.
\]
But then 
\[
s\le b_{\gamma^{+}}(x)-b_{\gamma^{+}}(\eta_{x}^{-}(s))\le d(x,\eta^{-}(s))=s
\]
and thus $b_{\gamma^{+}}(\eta_{x}^{-}(s))=b_{\gamma^{+}}(x)-s$. 

This also implies that for $t,s\ge0$ it holds 
\begin{eqnarray*}
b_{\gamma^{+}}(\eta_{x}^{+}(t))-b_{\gamma^{+}}(\eta_{x}^{-}(s)) & = & s+t\\
 & \le & d(\eta_{x}^{-}(s),\eta_{x}^{+}(t))\\
 & \le & d(\eta_{x}^{-}(s),x)+d(x,\eta_{x}^{+}(t))=s+t.
\end{eqnarray*}
Therefore, $\eta_{x}$ is a line bi-asymptotic to $\gamma$.
\end{proof}
\begin{lem}
Through each point there is exactly one line parallel to $\gamma$.
\end{lem}
\begin{proof}
By non-branching there is at most one bi-asymptotic line through each
$x$. Indeed, assume $\eta$ is a line through $x$ such that $b_{\gamma^{+}}$
is affine along $\eta$ and let $\tilde{\eta}$ be a ray which is
asymptotic to $\gamma^{+}$. Then it is easy to see that $b_{\gamma^{+}}$
is affine on $\eta'=\eta^{-}\cup\tilde{\eta}$ and thus $d(\eta^{-}(t),\tilde{\eta}(s))=t+s$.
But then by non-branching assumption we must have $\eta=\eta'$. The
same argument also show that $\eta$ is the unique line bi-asymptotic
to $\gamma$ that starts at $\eta(t)$ for any $t\in\mathbb{R}$.
Hence $\eta$ is parallel to $\gamma$.
\end{proof}
\begin{lem}
\label{lem:moving-isometry}For any $x,y\in X$ and $t,s,a\in\mathbb{R}$
it holds 
\[
d(\eta_{x}(t+a),\eta_{y}(s+a))=d(\eta_{x}(t),\eta_{y}(s)).
\]
\end{lem}
\begin{proof}
Observe that uniqueness of the lines $\eta_{x}$ and $\eta_{y}$ through
$x$ and resp. $y$ implies that for any $t_{n}\to\infty$ the sequences
$(\eta_{n})_{n\in\mathbb{N}}$ and $(\xi_{n})_{n\in\mathbb{N}}$ connecting
$x$ and $\gamma(t_{n})$ and resp. $y$ and $\gamma(t_{n})$ converge
to $\eta$ and resp. $\xi$. Thus we can apply Preposition \ref{prop:Busemann-contraction}
either with the ray $\gamma^{+}$ or with the ray $\gamma^{-}$ to
conclude 
\[
d(\eta(t),\xi(s))=d(\eta(t+a),\xi(s+a)).
\]
\end{proof}
This means that moving along the lines induces an isometry. In particular,
all level sets of $b_{\gamma^{+}}$ are isometric with isometry generated
by moving along the parallel lines.
\begin{cor}
Assume $(X,d)$ is a Busemann concave proper metric space. If through
every point $x\in X$ there is a line connecting $x$ with some fixed
$x_{0}$ then $(X,d)$ is homogeneous, i.e. for every $x,y\in X$
there is an isometry $g_{xy}\in\operatorname{Isom}(X,d)$ such that
$g_{xy}(x)=y$. 
\end{cor}
We are now able to prove the bi-Lipschitz splitting theorem.
\begin{thm}
[Splitting Theorem]\label{thm:Splitting}Let $(X,d)$ a complete
Busemann concave space and assume $X$ contains a line $\gamma:\mathbb{R}\to X$.
Then through every $x\in X$ there is a unique line parallel $\eta$
to $\gamma$ and $(X,d)$ is bi-Lipschitz to a metric space $(X'\times\mathbb{R},\tilde{d})$
where $\tilde{d}$ is a product metric. 
\end{thm}
\begin{rem*}
The proof shows that $X'$ is a subset of $X$ but it is not clear
whether it can be chosen to be convex/totally geodesic w.r.t. $d$
and $\tilde{d}$. In case $X'$ is totally geodesic w.r.t. $d$ the
product space can be chosen to be Busemann concave.
\end{rem*}
\begin{proof}
We claim that $(X,d)$ is bi-Lipschitz to $(X'\times\mathbb{R},\tilde{d})$
where $X'=b_{\gamma^{+}}^{-1}(0)$ and 
\[
\tilde{d}((x,t),(y,s))=d(x,y)+|t-s|.
\]
 This is obviously a product metric on $X'\times\mathbb{R}$. Now
let $\Phi:X\to(\mathbb{R}\to X)$ assign to each $x\in X$ the unique
line $\eta_{x}$ parallel to $\gamma$ such that $b_{\gamma^{+}}(\eta_{x}(0))=0$.
We claim that the map 
\[
\Psi(x)=(\eta_{x}(0),b_{\gamma^{+}}(x))
\]
 is a bi-Lipschitz homeomorphism between $(X,d)$ and $(X'\times\mathbb{R},\tilde{d})$.
Note that $\Psi$ is obviously bijective, so that for simplicity of
notation we identify $X$ and $(X'\times\mathbb{R})$ set-wise and
assume $(x,t),(y,s)$ are points living in $X$. 

Using the triangle inequality of $d$ we get 
\[
d((x,t),(y,s))\le d((x,t),(x,s))+d((x,s),(y,s))=|t-s|+d(x,y)=\tilde{d}((x,t),(y,s)).
\]
 which implies that $\Psi^{-1}$ is $1$-Lipschitz.

We claim that $d(x,y)+|t-s|\le3d((x,t),(y,s)).$ This would imply
that $\Psi$ is $3$-Lipschitz and finish the proof.

To show the claim note that $b_{\gamma^{+}}$ is $1$-Lipschitz so
that 
\[
|t-s|=|b_{\gamma^{+}}(x,t)-b_{\gamma^{+}}(y,s)|\le d((x,t),(y,s))
\]
and from triangle inequality
\begin{eqnarray*}
d(x,y) & \le & d((x,t),(y,s))+d((y,s),(y,t)\\
 & = & d((x,t),(y,s))+|t-s|.
\end{eqnarray*}
Combining we obtain 
\[
d(x,y)+|t-s|\le d((x,t),(y,s))+2|t-s|\le3d((x,t),(y,s)).
\]

Note that it is possible to change $\tilde{d}$ by any metric product
of $(X',d)$ and $(\mathbb{R},|\cdot-\cdot|)$ as any two norms on
$\mathbb{R}^{2}$ are bi-Lipschitz with Lipschitz constants only depending
on the two norms. In particular, if $X'$ was convex w.r.t. $d$ then
one may choose the $L^{2}$-product so that $X'\times_{2}\mathbb{R}$
is Busemann concave.
\end{proof}
If the Hausdorff measure is non-trivial and the space is ``Riemannian-like'',
then it is possible to show that $X'$ is indeed convex and $(X,d)$
is isometric to the $L^{2}$-product of $X'$ and the real line.
\begin{cor}
Assume, in addition, that $\mathcal{H}^{n}$ is non-trivial and $(X,d,\mathcal{H}^{n})$
is infinitesimally Hilbertian then $(X,d)$ is isometric to $X'\times_{2}\mathbb{R}$.
\end{cor}
\begin{proof}
This follows from the proof of Gigli's splitting theorem for $RCD$-spaces
\cite{Gigli2013a}. Just note that the ``gradient flow'' of the
Busemann function is just the isometry induced by moving along the
lines and this isometry also preserves the Hausdorff measure. Using
\cite[Theorem 5.23]{Gigli2013a} one shows that $X'$ is totally geodesic.
Furthermore, $L^{2}$-products of infinitesimally Hilbertian Busemann
concave spaces are also infinitesimally Hilbertian (compare with \cite[Theorem 6.1]{Gigli2013a}). 
\end{proof}
As mentioned above if $(X,d)$ is Busemann concave and angles are
well-defined then $(X,d)$ is an Alexandrov space. In Gigli's proof
of the splitting theorem, pointwise angles are replaced by ``smoothed''
angles. More precisely, instead of looking at two intersecting geodesic
in $X$ one can look at intersecting geodesics in the Wasserstein
space $\mathcal{P}_{2}(X)$. If the geodesics are pointwise absolutely
continuous measures then being infinitesimally Hilbertian shows that
there is a notion of angle. Thus one might ask.
\begin{problem}
Assume $(X,d,\mathcal{H}^{n})$ is an infinitesimally Hilbertian Busemann
concave metric measure space and $\mathcal{H}^{n}$ is non-trivial.
Is $(X,d)$ a non-negatively curved Alexandrov space? 
\end{problem}

\subsection{Tangent cones\label{subsec:Tangent-cones}}

Let $\Gamma_{x}$ be the set of maximal unit speed geodesics starting
at $x$. The \emph{pre-tangent cone} $\hat{T}_{x}X$ at $X$ is defined
as the set $\Gamma_{x}\times[0,\infty)$ such that the points $(\gamma,0)$
are identified. On $\hat{T}_{x}M$ we define a metric $d_{x}$ as
follows: Given geodesics $\gamma,\eta\in\Gamma_{x}$ there is an interval
$I=[0,a]$ such that $\gamma,\eta$ are both defined on $I$. Then
define a metric $d_{x}$ on $\hat{T}_{x}X$ by 
\[
d_{x}((\gamma,s),(\eta,t))=\sup_{r\in[0,1],\max\{rs,rt\}\le a}\frac{d(\gamma(rs),\eta(rt))}{r}.
\]

\begin{lem}
If $(X,d)$ is Busemann concave then $d_{x}$ is a well-defined metric
on $\hat{T}_{x}X$. Furthermore, it holds $d_{x}((\gamma,\lambda s),(\eta,\lambda t))=\lambda d_{x}((\gamma,s),(\eta,t))$.
\end{lem}
\begin{proof}
$d_{x}$ is obviously non-negative and symmetric. The function $r\mapsto\frac{d(\gamma(sr),\eta(tr))}{r}$
is non-increasing by Busemann concavity so that the supremum in the
definition of $d_{x}$ is actually a limit w.r.t. $r\to0$. This also
implies that the triangle inequality holds. Also note that $d_{x}((\gamma,s),(\eta,t))=0$
implies that $d(\gamma(rs),\eta(rt))=0$ so that $\gamma(rs)=\eta(rt)$.
As $\gamma$ and $\eta$ are unit speed geodesics starting at $x$
we must have $s=t$ so that either $s=t=0$ or $\gamma\equiv\eta$,
i.e. $d_{x}$ is definite.
\end{proof}
It is possible to define an \emph{exponential map} directly from the
pre-tangent cone: Let $U_{x}\subset\hat{T}_{x}X$ such that $(\gamma,t)\in U_{x}$
iff $\gamma(t)$ is defined. Then we define the exponential map $\exp_{x}:U\to X$
by 
\[
\exp_{x}(\gamma,t)=\gamma(t).
\]
Note that by definition, $\exp_{x}$ is onto. By Busemann concavity
and the definition of $d_{x}$ it is not difficult to show that $\exp_{x}$
is $1$-Lipschitz, i.e. $d_{x}(v,w)\ge d(\exp_{x}(v),\exp_{x}(w))$. 
\begin{defn}
[Tangent cone] The \emph{tangent cone} $(T_{x}X,d_{x})$ at $x$
is defined as the metric completion of $(\hat{T}_{x}X,d_{x})$. 
\end{defn}
The tangent cone $(T_{x}X,d_{x})$ is not necessarily the (pointed)
Gromov-Hausdorff limit of the blow ups $(X,\frac{1}{\lambda}d,x)$
at $x$. Indeed, if $(X,d)$ is compact and the tangent cone at some
point is not locally compact then it cannot be the Gromov-Hausdorff
limit of blow ups. An example is given by 
\[
K_{c}^{p}=\left\{ (x_{i})\in\ell^{p}\,|\,\sum c^{i}x_{i}^{p}\le1\right\} 
\]
where $c>1$ and $p\in(1,\infty)$. The set $K_{c}^{p}$ is a compact
convex subset of $\ell^{p}$, but the tangent cones at points $(x_{i})$
with $0<\sum c^{i}x_{i}^{p}<1$ are isometry to $\ell^{p}$. However,
if the blow-ups are precompact then their limit is uniquely given
by the tangent cone.
\begin{lem}
If $(X,\frac{1}{\lambda_{n}}d,x)$ converges in the Gromov-Hausdorff
topology then the limit equals $(T_{x}X,d_{x})$. In particular, if
$\{(X,\frac{1}{\lambda}d,x)\}_{\lambda\in(0,1]}$ is precompact then
the limit as $\lambda\to0$ exists and $(T_{x}X,d_{x})$ is its unique
GH-limit. In particular, the tangent cone is the (blow up) tangent
space at $x$.
\end{lem}
More generally, if the tangent cone $(T_{x}M,d_{x})$ was locally
compact we obtain the following.
\begin{lem}
If $(T_{x}X,d_{x})$ is locally compact then for each $r>0$ the sequence
$\{(B_{r}(x),\frac{1}{\lambda}d)\}_{\lambda\in(0,1]}$ is precompact
with respect to the Gromov-Hausdorff topology. In particular, $(T_{x}X,d_{x})$
is the (unique) pointed Gromov-Hausdorff limit of $\{(X,\frac{1}{\lambda}d,x)\}_{\lambda\in(0,1]}$
as $\lambda\to0$.
\end{lem}
\begin{proof}
By scale invariance of $T_{x}X$ any bounded subset is precompact.
Let $U_{r}$ be all $(\gamma,t)\in U$ with $t\le r$. One can define
a scaled exponential map $\exp_{x}^{\lambda}(\gamma,t)=\gamma(\lambda t)$.
Then $\exp_{x}^{\lambda}$ maps $(U_{r},d_{x})$ onto $(B_{r}^{\lambda}(x),\frac{1}{\lambda}d)$
and is $1$-Lipschitz where $B_{r}^{\lambda}(x)$ is the $(\frac{1}{\lambda}d)$-ball
of radius $r$ with center $x$. By assumption $U_{r}$ is bounded
and thus precompact in $(T_{x}X,d_{x})$. Therefore, for each $\epsilon>0$
there is an $N(\epsilon)<\infty$ such that $U_{r}$ can be covered
by $N(\epsilon)$ $d_{x}$-balls of radius $\epsilon$ with center
in $U_{r}$. 

We claim that that $B_{r}^{\lambda}$ can be covered by at most $N(\epsilon)$
$(\frac{1}{\lambda}d)$-balls of radius $\epsilon$. Indeed, let $\{v_{1},\ldots,v_{N(\epsilon)}\}$
be the centers of the $d_{x}$-balls of radius $\epsilon$. This means
for all $v\in U_{r}$ it holds $\inf_{i=1}^{N(\epsilon)}d_{x}(v,v_{i})\le\epsilon$.
Set $x_{i}=\exp_{x}^{\lambda}v_{i}$. As $\exp_{x}^{\lambda}$ is
onto for each $x'\in B_{r}^{\lambda}(x)$ there is a $v\in(\exp_{x}^{\lambda})^{-1}(x')$.
Combining this with the $1$-Lipschitz property we obtain 
\begin{eqnarray*}
\inf_{i=1,\ldots,N(\epsilon)}d(x',x_{i}) & = & \inf_{i=1,\ldots,N(\epsilon)}d(\exp_{x}^{\lambda}(v),\exp_{x}^{\lambda}(v_{i}))\\
 & \le & \inf_{i=1,\ldots,N(\epsilon)}\lambda d_{x}(v,v_{i})\le\lambda\epsilon.
\end{eqnarray*}
Hence, $\{B_{\epsilon}^{\lambda}(x_{i})\}_{i=1}^{N(\epsilon)}$ covers
$B_{r}^{\lambda}(x)$. By definition $\diam U_{r},\diam B_{r}^{\lambda}\le2r$
so that Gromov's precompactness theorem implies that $\{(B_{r}^{\lambda}(x),\frac{1}{\lambda}d)\}$
is precompact. Together with the previous lemma we see that the limit
has to be $(\closure U_{r},d_{x})=(B_{r}^{T_{x}M},d_{x})$. 
\end{proof}
\begin{cor}
If $(X,d)$ is a complete Busemann concave (local) doubling metric
space then $(T_{x}M,d_{x})$ is locally compact and the unique limit
of the blowups $(X,\frac{1}{\lambda}d,x)_{\lambda\in(0,\epsilon]}$. 
\end{cor}
By \cite{LeDonne2011} it can be shown that for topologically and
measure-theoretically almost all points $(T_{x}M,d_{x})$ is a (finite
dimensional) Carnot group. However, the limit does not have to be
Busemann concave so that Proposition \ref{prop:Carnot-Busemann} cannot
be applied. Nevertheless, the results above show that a homogeneous
tangent cone is necessarily both weakly Busemann concave and the central
contraction is actually affine. Both should imply that it has to be
a Banach space. This, in particular, would imply that the theory of
Busemann concave spaces are metric generalizations of Finsler manifolds.

\subsection{Hausdorff measure, doubling and Poincar\'e\label{subsec:Hausdorff-doubling-Poincare}}

The Hausdorff measure is a natural measure associated to a metric
space. For finite dimensional Alexandrov spaces it is known that there
is an integer $n$ such that the $n$-dimensional Hausdorff measure
is non-trivial, i.e. non-zero and locally finite \cite{Burago1992}.
Furthermore, if the space is non-negatively curved then this measure
is doubling. For general Busemann concave spaces, we currently cannot
show that the Hausdorff measure is non-trivial if the space is finite
dimensional w.r.t. to any meaningful dimension definition. However,
we will show that if the Hausdorff measure is non-trivial then it
is doubling and satisfies a $(1,1)$-Poincar\'e inequality. Furthermore,
it also satisfies the measure contraction property which is a (very)
weak form of non-negative Ricci curvature. It is likely that further
analysis shows that Busemann concavity implies that the space satisfies
$CD(0,n)$, i.e. Busemann concave spaces have non-negative $n$-dimensional
Ricci curvature in the sense of Lott-Sturm-Villani. 

Let $\delta>0$ and $S$ be a subset of $X$. Define 
\[
\mathcal{H}_{\delta}^{n}(S)=C_{n}\inf\left\{ \sum_{i\in\mathbb{N}}(\frac{1}{2}\diam U_{i})^{n}\,|\,S\subset\bigcup_{i\in\mathbb{N}}U_{i},\diam U_{i}<\delta\right\} 
\]
where $C_{n}$ is a constant such that if $(X,d)=(\mathbb{R}^{n},d_{\operatorname{Euclid}})$
the measure $\mathcal{H}^{n}$ equals the Lebesgue measure on the
$n$-dimensional Euclidean space. Note that $\mathcal{H}_{\delta}^{n}(S)$
is decreasing in $\delta$ so that we can define the \emph{$n$-dimensional
Hausdorff measure} of $S$ as follows 
\[
\mathcal{H}^{n}(S)=\sup_{\delta>0}\mathcal{H}_{\delta}^{n}(S)=\lim_{\delta\to0}\mathcal{H}_{\delta}^{n}(S).
\]

From the definition it follows that $\mathcal{H}^{n}$ is an outer
measure. Furthermore, one can show that each Borel set of $X$ is
$\mathcal{H}^{n}$-measurable and thus $\mathcal{H}^{n}$ a Borel
measure.

Note that we have the following: if for some $n$ it holds $\mathcal{H}^{n}(S)<\infty$
then $\mathcal{H}^{n'}(S)=0$ for all $n'>n$. Also if $\mathcal{H}^{n}(S)>0$
then $\mathcal{H}^{n'}(S)=\infty$ for all $0<n'<n$. In particular,
for each $S$ there is at most one $n$ with $0<\mathcal{H}^{n}(S)<\infty$.
Therefore, we can assign to each bounded set a number called \emph{Hausdorff
dimension} 
\begin{eqnarray*}
\dim_{H}S & = & \inf\{n\in[0,\infty)\,|\text{\,}\mathcal{H}^{n}(S)=0\}\\
 & = & \sup\{n\in[0,\infty)\,|\text{\,}\mathcal{H}^{n}(S)=\infty\}
\end{eqnarray*}
with conventions $\inf\varnothing=\infty$ and $\sup\varnothing=0$.
Now denote the \emph{local Hausdorff dimension at $x$} by
\[
\dim_{H}X(x)=\inf_{x\in U\,\operatorname{open}}\dim_{H}U.
\]

Given points $x,y\in X$ we can choose a geodesic $\gamma_{xy}$ connecting
$x$ and $y$ such that for any $t\in[0,1]$ the map
\[
\Phi:(y,x,t)\mapsto\gamma_{xy}(t)
\]
is a measurable function. Without loss of generality it is possible
to choose $\Phi$ in a symmetric way, i.e. $\Phi(y,x,t)=\Phi(x,y,1-t)$.

Let $\Omega$ be some subset of $X$ and define $\Omega_{t}=\Phi(\Omega,x,t)$.
Denote its inverse by $g:\Omega_{t}\to\Omega$. Note that this map
is onto and Busemann concavity implies it is $t^{-1}$-Lipschitz.
Let $\{U_{i}\}_{i\in\mathbb{N}}$ be a $\delta$-cover of $\Omega_{t}$
then $\{g(U_{i})\}_{i\in\mathbb{N}}$ is a $t^{-1}\delta$-cover of
$\Omega$. Furthermore, it holds
\[
\sum(\diam g(U_{i}))^{n}\le\frac{1}{t^{n}}\sum(\diam U_{i})^{n}
\]
and hence

\[
\mathcal{H}_{t^{-1}\delta}^{n}(\Omega)\le\frac{1}{t^{n}}\mathcal{H}_{\delta}^{n}(\Omega_{t}).
\]
Taking the limit as $\delta\to0$ on both sides we see that 
\begin{equation}
\mathcal{H}^{n}(\Omega)\le\frac{1}{t^{n}}\mathcal{H}^{n}(\Omega_{t}).\label{eq:Hausdorff-contraction}
\end{equation}

Note that this implies that if $\mathcal{H}^{n}(B_{r}(x))<\infty$
then $\mathcal{H}^{n}(\Omega)<\infty$ for all bounded $\Omega$.
Indeed, there is a $t>0$ depending only on $\Omega,x$ and $r$ such
that $\Omega_{t}=\Phi(\Omega,x,t)\subset B_{r}(x)$. 
\begin{lem}
In a Busemann concave space $(X,d)$ the Hausdorff dimension of bounded
open subsets is equal to a fixed number $n\in\mathbb{N}\cup\{\infty\}$
which depends only the space itself. In particular, $\dim_{H}X(x)\equiv\mbox{const}$.
\end{lem}
\begin{proof}
Let $\Omega,\Omega'$ be two bounded subsets with non-empty interior.
Then there is an $x\in\Omega,x'\in\Omega'$ and $r,t>0$ such that
$\Omega_{t}\subset B_{r}(x')\subset\Omega'$ and $\Omega_{t}^{'}\subset B_{r}(x)\subset\Omega$
where $\Omega_{t}=\Phi(\Omega,x',t)$ and $\Omega_{t}^{'}=\Phi(\Omega',x,t)$.
Hence, by \ref{eq:Hausdorff-contraction} it holds 
\[
\mathcal{H}^{n}(\Omega)\le\frac{1}{t^{n}}\mathcal{H}^{n}(\Omega^{'})
\]
and 
\[
\mathcal{H}^{n}(\Omega')\le\frac{1}{t^{n}}\mathcal{H}(\Omega).
\]
Now it is easy to see that $\dim_{H}\Omega=\dim_{H}\Omega'$ and that
this number equals the local Hausdorff dimension at $x$.
\end{proof}
In case the Hausdorff dimension is finite, it is still not clear if
the corresponding measure is non-trivial, i.e. $0<\mathcal{H}^{n}(B_{r}(x))<\infty$
for some $x\in X$ and $r>0$. However, if the $n$-dimensional Hausdorff
measure is non-trivial, then the space enjoys nice properties.
\begin{prop}
\label{prop:Busemann-MCP-et-al}Assume $(X,d)$ is a complete Busemann
concave metric space admitting a non-trivial Hausdorff measure. Then
$(X,d,\mathcal{H}^{n})$ satisfies the measure contraction property
$MCP(0,n)$, the Bishop-Gromov volume comparison $BG(0,n)$ and a
(weak) $(1,1)$-Poincar\'e inequality. In particular, $\mathcal{H}^{n}$
is a doubling measure with doubling constant $2^{n}$. 
\end{prop}
\begin{rem*}
We refer to \cite{Ohta2007a} for the exact definition of the measure
contraction property and to \cite{Bjorn2011,Heinonen2015} for definitions
of Poincar\'e and doubling conditions and their influence on the
geometry and analysis of metric spaces.
\end{rem*}
\begin{proof}
Since $\Phi$ is measurable we can apply \cite[Lemma 2.3]{Ohta2007a}.

The Bishop-Gromov volume comparison and the doubling property follows
once we notice that $\Omega=B_{R}(x)$ implies that $\Omega_{\frac{r}{R}}\subset B_{r}(x)$.
Equation \ref{eq:Hausdorff-contraction} then implies 
\[
\frac{\mathcal{H}^{n}(B_{R}(x))}{\mathcal{H}^{n}(B_{r}(x)).}\le\frac{R^{n}}{r^{n}}=\frac{\mathcal{V}_{0,n}(R)}{\mathcal{V}_{0,n}(r)}\quad\mbox{for }0<r<R.
\]
In particular, $r\mapsto\frac{\mathcal{H}^{n}(B_{r}(x))}{\mathcal{V}_{0,n}(r)}$
is non-increasing. 

A standard argument implies that a weak $(1,1)$-Poincar\'e inequality
holds (see e.g. \cite[Lemma 3.3]{Hua2010}), i.e. it holds
\begin{eqnarray*}
\int_{B_{r}(x)}|u-u_{B_{r}(x)}|d\mathcal{H}^{n} & \le & 2^{n+1}r\int_{B_{3r}(x)}g_{u}d\mathcal{H}^{n}
\end{eqnarray*}
where $g_{u}$ is a weak upper gradient of $u$. 

We sketch the argument given in \cite[Lemma 3.3]{Hua2010}: It suffices
to assume $u\in\operatorname{Lip}(X,d)$ and $g_{u}=\operatorname{lip}u$
where $\operatorname{lip}f$ is the local Lipschitz constant of $u$.
Set $B=B_{r}(x)$ and $u_{B}=\frac{1}{\mathcal{H}^{n}(B)}\int_{B}ud\mathcal{H}^{n}$.
Then for every geodesic $\gamma_{xy}:[0,1]\to X$ connecting $x,y\in B_{r}(x)$
it holds 
\[
|u(y)-u(z)|\le d(y,z)\int_{0}^{1}g_{u}(\gamma_{yz}(t))dt.
\]
Thus 
\begin{eqnarray*}
\int_{B}|u-u_{B}|d\mathcal{H}^{n} & \le & \frac{1}{\mathcal{H}^{n}(B)}\int_{B}\int_{B}|u(y)-u(z)|d\mathcal{H}^{n}(z)d\mathcal{H}^{n}(y)\\
 & \le & \frac{2r}{\mathcal{H}^{n}(B)}\int_{B}\int_{B}\int_{0}^{1}g_{u}(\gamma_{yz}(t))dtd\mathcal{H}^{n}(z)d\mathcal{H}^{n}(y)\\
 & \le & \frac{4r}{\mathcal{H}^{n}(B)}\int_{B}\int_{B}\int_{\frac{1}{2}}^{1}g_{u}(\gamma_{yz}(t))dtd\mathcal{H}^{n}(z)d\mathcal{H}^{n}(y)
\end{eqnarray*}
where the last inequality follows by choosing the geodesics $\gamma_{yz}$
in a symmetric way. Note that the measure contraction property implies
\[
\int_{B}f(\gamma_{xy}(t))d\mathcal{H}^{n}(y)\le\frac{1}{t^{n}}\int_{B_{tr}(x)}f(z)\mathcal{H}^{n}(z)
\]
for all non-negative $f\in L_{\operatorname{loc}}^{\infty}(X,\mathcal{H}^{n})$
and $t\in(0,1]$, see \cite[Equations (2.2)]{Ohta2007a}. Therefore,
we obtain 
\begin{eqnarray*}
\int_{B}|u-u_{B}|d\mathcal{H}^{n} & \le & \frac{4r}{\mathcal{H}^{n}(B)}\int_{B}\int_{\frac{1}{2}}^{1}\int_{B_{2r}(y)}g_{u}(\gamma_{yz}(t))d\mathcal{H}^{n}(z)dtd\mathcal{H}^{n}(y)\\
 & \le & \frac{4r}{\mathcal{H}^{n}(B)}\int_{B}\int_{\frac{1}{2}}^{1}\frac{1}{t^{n}}\int_{B_{2tr}(y)}g_{u}(w)d\mathcal{H}^{n}(w)dtd\mathcal{H}^{n}(y)\\
 & \le & \frac{4r}{\mathcal{H}^{n}(B)}\int_{B}\int_{\frac{1}{2}}^{1}\frac{1}{t^{n}}\int_{B_{3r}(x)}g_{u}(w)d\mathcal{H}^{n}(w)dtd\mathcal{H}^{n}(y)\\
 & \le & 2^{n+1}r\int_{B_{3r}(x)}g_{u}d\mathcal{H}^{n}.
\end{eqnarray*}
\end{proof}
\begin{rem*}
In order to prove the Bishop inequality it remains to show that 
\[
\lim_{r\to0}\frac{\mathcal{H}^{n}(B_{r}(x))}{r^{n}}=\omega_{n}
\]
where $\omega_{n}=\mathcal{V}_{0,n}(1)$. This holds at points where
the (blow-up) tangent space is isometric to an $n$-dimensional normed
space as the Hausdorff measure of balls equals the volume of the $n$-dimensional
Euclidean balls of same radius and the pointed Gromov-Hausdorff convergence
is compatible with the measured Gromov-Hausdorff convergence if the
reference measures are non-collapsing Hausdorff measures of the same
dimension.
\end{rem*}

\subsection{Bonnet-Myers theorem\label{subsec:Bonnet-Myers-theorem}}

Throughout this section we assume that geodesics are parametrized
by arc length, i.e. they are unit speed geodesics. This will simplify
some of the proofs below. 

Recall a fact on triangles in $S^{2}$: Let $a,b,c\in(0,\pi]$ with
$a+b+c\le2\pi$. Then there is a triangle formed by unit speed geodesics
$\tilde{\gamma},\tilde{\eta},\tilde{\xi}$ of length $a,b,c$ with
$\tilde{\gamma}_{0}=\tilde{\eta}_{0}$, $\tilde{\gamma}_{a}=\tilde{\xi}_{0}$
and $\tilde{\eta}_{b}=\tilde{\xi}_{c}$. Furthermore, if $a=b\ge\frac{\pi}{2}$
and $a+b+c=2\pi$ then 
\[
d_{S^{2}}(\tilde{\gamma}_{\frac{\pi}{2}},\tilde{\eta}_{\frac{\pi}{2}})=\pi
\]
and $\tilde{\eta}_{0}=\tilde{\eta}_{0}$ is a midpoint of the pair
$(\tilde{\gamma}_{\frac{\pi}{2}},\tilde{\eta}_{\frac{\pi}{2}})$.
The proof of the Bonnet-Myers theorem relies heavily on this rigidity.

We say that the complete geodesic metric space $(X,d)$ has \emph{Busemann
curvature bounded below by $1$} if for all unit speed geodesics $\gamma,\eta,\xi$
in $X$ of length $a,b,c$ with $\gamma_{0}=\eta_{0}$, $\gamma_{a}=\xi_{0}$
and $\eta_{b}=\xi_{c}$ such that there is a corresponding comparison
triangle in $S^{2}$ of length $a,b,c\ge0$ then it holds 
\[
d(\gamma_{ta},\eta_{tb})\ge d_{S^{2}}(\tilde{\gamma}_{ta},\tilde{\eta}_{tb})
\]
for all $t\in[0,1]$. 

Using a slightly different notion of positive curvature Kann \cite{Kann1961}
obtained a Bonnet-Myers theorem for two dimensional $G$-spaces which
are positively curved in his sense. Note that his proof heavily relies
on the notion of two dimensionality as well as local extendability
of geodesics. The proof of the Bonnet-Myers theorem below is inspired
by the one for $MCP(K,N)$-spaces \cite[Section 4]{Ohta2007a}. The
main idea is to replace the density estimates by length estimates.
However, the technique is quite different and some steps are easier
to prove in the current setting. 

Before we start we need the following characterization of non-branching
spaces that are not $1$-dimensional: A geodesic metric space is said
to be \emph{not $1$-dimensional }if for any unit speed geodesic $\gamma:[0,a]\to X$
and $\epsilon>0$ there is a $y\in B_{a}(x)$ with $d(\gamma_{a},y)<\epsilon$
but $y\notin\gamma_{[0,a]}$. 

For non-branching $1$-dimensional spaces one gets the following rigidity.
We leave the details to the interested reader, compare also with \cite[Theorem (9.6)]{Busemann1955}.
\begin{lem}
If $(X,d)$ is $1$-dimensional and non-branching then it is isometric
to a closed interval $I\subset X$ or a circle $S_{\lambda}^{1}$
of length $\lambda>0$.
\end{lem}
Until the end of this section we assume that $(X,d)$ has Busemann
curvature bounded below by $1$. 
\begin{lem}
\label{lem:contract}If $\gamma,\eta:[0,\pi-a]\to X$ are two unit
speed geodesics starting at $x\in X$ with $d(\gamma_{\frac{\pi}{2}},\eta_{\frac{\pi}{2}})<\pi$
and $a\in(0,\frac{\pi}{2}]$ then for any $s\in[a,\frac{\pi}{2}]$
it holds $d(\gamma_{\pi-s},\eta_{\pi-s})<2s$.
\end{lem}
\begin{proof}
Not first that 
\[
\limsup_{s\to\frac{\pi}{2}^{\text{\textsuperscript{-}}}}d(x,\gamma_{\pi-s})+d(x,\gamma_{\pi-s})+d(\gamma_{\pi-s},\eta_{\pi-s})<2\pi.
\]
Thus 
\[
\limsup_{s\to\frac{\pi}{2}^{-}}(d(\gamma_{\pi-s},\eta_{\pi-s})-2s)<0.
\]
In particular, for $s$ sufficiently close to $\frac{\pi}{2}$ it
holds $d(\gamma_{\pi-s},\eta_{\pi-s})<2s$. 

Assume the statement was not true. Then there is a largest $s_{0}\in[a,\frac{\pi}{2})$
with $d(\gamma_{\pi-s_{0}},\eta_{\pi-s_{0}})=2s_{0}$. The assumptions
imply that there is a comparison triangle of the triangle formed by
$(x,\gamma_{\pi-s_{0}},\eta_{\pi-s_{0}})$ such that for some unit
speed geodesics $\tilde{\gamma},\tilde{\eta}:[0,\pi-s_{0}]\to S^{2}$
it holds $d(\gamma_{\pi-s_{0}},\eta_{\pi-s_{0}})=d(\tilde{\gamma}_{\pi-s_{0}},\tilde{\eta}_{\pi-s_{0}})$
and 
\[
d(\gamma_{t},\eta_{t})\ge d_{S^{2}}(\tilde{\gamma}_{t},\tilde{\eta}_{t})\qquad\mbox{for }t\in[0,\pi-s_{0}].
\]
However, the comparison triangle satisfies 
\[
d_{S^{2}}(\tilde{\gamma}_{\frac{\pi}{2}},\tilde{\eta}_{\frac{\pi}{2}})=\pi
\]
which would contradict the assumptions $d(\gamma_{\frac{\pi}{2}},\eta_{\frac{\pi}{2}})<\pi$.
\end{proof}
\begin{lem}
\label{lem:approx}Assume $(X,d)$ is not $1$-dimensional. Then for
$\epsilon>0$ and all unit speed geodesics $\gamma,\eta:[0,\pi-s]\to X$
with $s\in(0,\frac{\pi}{2})$ there is a unit speed geodesic $\eta^{(\epsilon)}:[0,\pi-s_{\epsilon}]\to X$
such that $s_{\epsilon}\ge s$, $d(\eta_{\pi-s},\eta_{\pi-s_{\epsilon}}^{\epsilon})<\epsilon$
and 
\[
d(\gamma_{\frac{\pi}{2}},\eta_{\frac{\pi}{2}}^{\epsilon})<\pi.
\]
\end{lem}
\begin{proof}
If $d(\gamma_{\frac{\pi}{2}},\eta_{\frac{\pi}{2}})<\pi$ then we can
choose $\tilde{\eta}=\eta$. Assume $d(\gamma_{\frac{\pi}{2}},\eta_{\frac{\pi}{2}})=\pi$.
Since $(X,d)$ is not $1$-dimensional there is a sequence of points
$(y_{n})_{n\in\mathbb{N}}$ in $B_{\pi-s}(x)$ such that $y_{n}\to\eta_{\pi-s}$
and $y_{n}\ne\eta(\pi-s_{n})$ where $s_{n}=\pi-d(x,y_{n})$. Let
$\eta^{n}:[0,\pi-s_{n}]\to X$ be a unit speed geodesic connecting
$x$ and $y_{n}$. Note by non-branching $\eta_{\frac{\pi}{2}}^{n}\ne\eta_{\frac{\pi}{2}}$.

We claim that $d(\gamma_{\frac{\pi}{2}},\eta_{\frac{\pi}{2}}^{n})<\pi$.
Indeed, equality would imply that $x$ is a midpoint of both $(\gamma_{\frac{\pi}{2}},\eta_{\frac{\pi}{2}})$
and $(\gamma_{\frac{\pi}{2}},\eta_{\frac{\pi}{2}}^{n})$ which is
not possible by the non-branching assumption. 
\end{proof}
In the following we write $\diam\varnothing=0$. Thus $\diam A=0$
implies that $A$ contains at most one element. 
\begin{cor}
\label{cor:contract}If $(X,d)$ is not $1$-dimensional and $\gamma,\eta$
are as in the lemma then 
\[
d(\gamma_{\pi-s},\eta_{\pi-s})\le2s
\]
for $s\in[0,\frac{\pi}{2}]$. In particular, $\diam\partial B_{\pi-s}(x)\le2s$
for $s\in[0,\frac{\pi}{2}]$.
\end{cor}
\begin{proof}
Assume first $s>0$. Using Lemma \ref{lem:approx} we find a sequence
of unit speed geodesics $\eta^{n}:[0,\pi-s_{n}]\to X$ with $s_{n}\ge s$,
$\eta_{\pi-s_{n}}^{n}\to\eta_{\pi-s}$ and $d(\gamma_{\frac{\pi}{2}},\eta_{\frac{\pi}{2}}^{n})<\pi$.
Thus we can apply Lemma \ref{lem:contract} to $\gamma|_{[0,\pi-s_{n}]}$
and $\eta^{n}$ and get 
\[
d(\gamma_{\pi-s_{n}},\eta_{\pi-s_{n}}^{n})<2s_{n}.
\]
Since $s_{n}\to s$ we see that 
\[
d(\gamma_{\pi-s},\eta_{\pi-s})=\lim_{n\to\infty}d(\gamma_{\pi-s_{n}},\eta_{\pi-s_{n}}^{n})\le\lim_{n\to\infty}2s_{n}=2s.
\]

The case $s=0$ is obtain via approximation. 
\end{proof}
Combining the results we obtain the Bonnet-Myers theorem. 
\begin{thm}
[Bonnet-Myers Theorem]\label{thm:Bonnet-Myers}Assume $(X,d)$ has
Busemann curvature bounded below by $1$. If $(X,d)$ is not $1$-dimensional
then the diameter of $X$ is at most $\pi$. 
\end{thm}
\begin{rem*}
By scaling one sees that $\diam X\le\frac{\pi}{\sqrt{k}}$ if the
space has Busemann curvature bounded below by $k>0$.
\end{rem*}
\begin{proof}
By Corollary \ref{cor:contract} we have $\diam\partial B_{\pi}(x)=0$
for all $x\in X$. In particular, $\partial B_{\pi}(x)$ contains
at most one element. 

If $\diam X\ge\pi$ then $\partial B_{\pi}(x)=\{x^{*}\}$ for some
$x,x^{*}\in X$. Assume by contradiction $\diam X>\pi$. Then for
some sufficiently small $\epsilon>0$ there is a unit speed geodesic
$\xi:[0,\pi+\epsilon]\to X$ starting at $x$. Since $d(x,\xi(\pi))=\pi$
we must have $\xi(\pi)=x^{*}$. However, $(X,d)$ is not $1$-dimensional
and geodesic so that there is a $\tilde{x}$ with $d(x,\tilde{x})=\pi$
and $d(\xi(\pi),\tilde{x})>0$ which contradicts the fact that $\diam\partial B_{\pi}(x)=0$.
\end{proof}
By the same arguments it is possible to show that if $(X,d)$ is not
$1$-dimensional then for all unit speed geodesics $\gamma:[0,a]\to X$
and $\eta:[0,b]\to X$ starting at $x$ with $a,b\in[0,\pi)$, $a+b\ge\pi$
the triangle formed by $(x,\gamma_{a},\eta_{b})$ has circumference
at most $2\pi$ with strict inequality unless $d(\gamma_{ta},\eta_{tb})=\pi$
where $ta+tb=\pi$. In particular, given three point in $X$ there
is always a corresponding comparison triangle in $S^{2}$. 

Using this we can prove the following along the lines of \cite[Section 5]{Ohta2007p}:
if $d(x,x^{*})=\pi$ for some $x,x^{*}\in X$ then for any $z\in X$
it holds 
\[
d(x,z)+d(z,x^{*})=\pi
\]
and there is a unique unit speed geodesic $\gamma$ connecting $x$
and $x^{*}$ such that $\gamma(d(x,z))=z$. This implies that if $X$
is compact then it is homeomorphic to a suspension of the space $(\partial B_{\frac{\pi}{2}}(x),d)$.
We leave the details to the interested reader. An alternative proof
can be done along the lines of \cite{Ketterer2015c}. For this note
that Euclidean cone $\Con(X)$ is Busemann concave. If $\diam X=\pi$
then $\Con(X)$ contains a line and by Theorem \ref{thm:Splitting}
it splits so that $X$ must be homeomorphic to a spherical suspension.
As in the proof of the splitting theorem, it is not clear whether
$\partial B_{\frac{\pi}{2}}(x)$ is totally geodesic.

\section{Uniformly smooth spaces}

In this section we define a form global non-negative curvature via
``smoothness'' assumption on the metric. Indeed, any Riemannian
manifold whose distance is uniformly smooth (see below) must have
non-negative sectional curvature on all planes spanned by tangent
vectors which are tangent to a ray. However, locally the distance
is $C^{\infty}$ and automatically uniformly smooth. 

The notion is inspired by the theory of Banach spaces: Uniform smoothness
of the norm implies that Busemann functions are well-defined and linear.
In particular they are given as duals of the corresponding vector
which represents the ray. In this section we want to use uniform smoothness
to show that Busemann functions are quasi-convex. A stronger condition,
called $p$-uniformly convex, will give convexity. From this one can
obtain by an argument of Cheeger-Gromoll \cite{Cheeger1972} that
space is an exhaustion of convex sets and can be retracted to a compact
totally geodesic subspace, i.e. a variant of the soul theorem. Note,
however, in the smooth setting this retract can have a non-empty boundary,
so that the compact retract should rather be called a weak soul. 

In the end of this section we try to give a local version of non-negative
curvature inspired by Gromov's characterization of non-negative curvature
in terms of inward equidistant movements of convex hypersurfaces. 

\subsection{Uniform smoothness and convexity of Busemann functions}
\begin{defn}
[Uniform smoothness] A geodesic metric space is said to be uniformly
smooth if there is an non-decreasing function $\rho:(0,\infty)\to[0,\infty)$
such that $\frac{\rho(\epsilon)}{\epsilon}\to0$ and for all $x,y,z\in X$
with 
\[
d(y,z)\le\epsilon\min\{d(x,y),d(x,z)\}
\]
it holds 
\[
d(x,m)\ge(1-\rho(\epsilon))\min\{d(x,y),d(x,z)\}
\]
for all midpoint $m$ of $y$ and $z$.
\end{defn}
We leave it to the interested reader to show that this is equivalent
to the usual definition of uniform smoothness in case $X$ is a Banach
space. A stronger variant is the so called $p$-uniform smoothness
for $p\in(1,2]$.
\begin{defn}
[$p$-uniform smoothness] A geodesic metric space is said to be $p$-uniformly
smooth if there is a $C>0$ such that for all $x,y,z\in X$ and it
holds 
\[
d(x,m)^{p}\ge\frac{1}{2}d(x,y)^{p}+\frac{1}{2}d(x,z)^{p}-\frac{C}{4}d(y,z)^{p}
\]
for all midpoints $m$ of $y$ and $z$.
\end{defn}
Note that by Clarkson's inequality every $L^{p}$-space is $p'$-uniformly
smooth for $p'=\min\{p,2\}$. Furthermore, the dual of a $q$-uniformly
convex Banach space is $p$-uniformly smooth with $\frac{1}{p}+\frac{1}{q}=1$.

The following was proved by Ohta \cite[Theorem 4.2]{Ohta2008}. 
\begin{lem}
Any Berwald space of non-negative flag curvature is $2$-uniformly
smooth.
\end{lem}
However, as every Berwald space of non-negative flag curvature is
affinely equivalent to a Riemannian manifold of non-negative curvature
\cite{Szabo1981,Szabo2006}, one can obtain most topological and geometric
properties directly from the affinely equivalent Riemannian manifold
(see \cite{Kell2015}).

It is not difficult to show that $p$-uniform smoothness implies uniform
smoothness. Indeed, one has 
\[
d(x,m)^{p}\ge(1-\tilde{\rho})A
\]
where $A=\frac{1}{2}d(x,y)^{p}+\frac{1}{2}d(x,z)^{p}$ and $A\frac{4}{C}\tilde{\rho}=d(y,z)^{p}$.
Because $A\ge\min\{d(x,y),d(x,z)\}^{p}$, we have 
\[
d(x,m)\ge(1-\rho)\min\{d(x,y),d(x,z)\}
\]
where $\rho=\min\{1,\tilde{\rho}\}$. We also have 
\[
d(y,z)\le\left(\frac{4}{C}\right)^{\frac{1}{p}}\rho^{\frac{1}{p}}\min\{d(x,y),d(y,z)\}.
\]
In particular, we may choose 
\[
\rho(\epsilon)=\min\{\frac{C}{4}\epsilon^{p},1\}
\]
to conclude.

An integral part of the soul theorem is the following function which
we call \emph{Cheeger-Gromoll function} (w.r.t. $x_{0}\in X$) 
\[
b_{x_{0}}(x)=\sup b_{\gamma}(x)
\]
where the supremum is taken over all rays starting at $x_{0}$.
\begin{prop}
Assume $(X,d)$ is uniformly smooth. Then any Busemann function $b_{\gamma}$
associated to a ray $\gamma$ is quasi-convex. In particular, all
Cheeger-Gromoll functions are quasi-convex.
\end{prop}
\begin{proof}
Fix distinct points $x,y\in X$ and assume $x,y\notin\gamma([t,\infty))$
for some large $t$. Let $\epsilon_{t}=\frac{d(x,y)}{\min\{d(x,\gamma_{t}),d(y,\gamma_{t})\}}$
and $m$ be a midpoint of $x$ and $y$. Then by uniform smoothness
it holds 
\begin{eqnarray*}
t-d(m,\gamma_{t}) & \le & t-(1-\rho(\epsilon_{t}))\min\{d(x,\gamma_{t}),d(y,\gamma_{t})\}.\\
 & = & \max\{t-d(x,\gamma_{t}),t-d(y,\gamma_{t})\}-\frac{\rho(\epsilon_{t})}{\epsilon_{t}}d(x,y).
\end{eqnarray*}
Note $t\to\infty$ implies $\epsilon_{t}\to0$ so that the rightmost
term converges to zero. But then
\begin{eqnarray*}
b_{\gamma}(m) & = & \lim_{t\to\infty}t-d(m,\gamma_{t})\\
 & \le & \lim_{t\to\infty}\max\{t-d(x,\gamma_{t}),t-d(y,\gamma_{t})\}\\
 & = & \max\{b_{\gamma}(x),b_{\gamma}(y)\}.
\end{eqnarray*}
Since the Cheeger-Gromoll functions are suprema of quasi-convex functions
they are quasi-convex as well.
\end{proof}
\begin{prop}
\label{prop:convex-Busemann}Assume $(X,d)$ is $p$-uniformly smooth.
Then Busemann function $b_{\gamma}$ associated to any ray $\gamma$
is convex. In particular, all Cheeger-Gromoll functions are convex.
\end{prop}
\begin{proof}
Fix distinct point $x,y\in X$ and assume $x,y\notin\gamma([t,\infty))$.
Then
\[
t-\frac{d(m,\gamma_{t})^{p}}{t^{p-1}}\le\frac{1}{2}\left(t-\frac{d(x,\gamma_{t})^{p}}{t^{p-1}}\right)+\frac{1}{2}\left(t-\frac{d(y,\gamma_{t})^{p}}{t^{p-1}}\right)-\frac{C}{4}\frac{d(x,y)^{p}}{t^{p-1}}.
\]
Note that the limit of the rightmost term converges to $0$ as $t\to\infty$
and by Lemma \ref{lem:fraction-busemann} 
\[
\lim_{t\to\infty}t-\frac{d(z,\gamma_{t})^{p}}{t^{p-1}}=\lim_{t\to\infty}t-d(z,\gamma_{t})=b_{\gamma}(z)
\]
for any $z\in X$. Combining these gives 
\[
b_{\gamma}(m)\le\frac{1}{2}b_{\gamma}(x)+\frac{1}{2}b_{\gamma}(y).
\]

Since the Cheeger-Gromoll functions are suprema of convex functions
they are convex as well.
\end{proof}
\begin{rem*}
By the same arguments the functions 
\begin{eqnarray*}
\tilde{b}_{x_{0}}(x) & = & \limsup_{t\to\infty}\sup_{y\in\partial B_{t}(x_{0})}t-d(x,y)\\
\hat{b}_{x_{0}}(x) & = & \limsup_{y_{n}\to\infty}d(x_{0},y_{n})-d(x,y_{n})
\end{eqnarray*}
are both quasi-convex or resp. both convex.
\end{rem*}
\begin{cor}
Assume $(X,d)$ is locally compact and $p$-uniformly smooth and geodesics
in $(X,d)$ can be extended locally. Then for any embedding line $\gamma:\mathbb{R}\to X$
\[
b_{\gamma\text{\textsuperscript{+}}}+b_{\gamma^{-}}=0.
\]
In particular, $b_{\gamma^{+}}$ is affine. If, in addition, $(X,d)$
is non-branching then $X$ is homeomorphic to $b_{\gamma^{+}}^{-1}(0)\times\mathbb{R}$.
\end{cor}
\begin{proof}
Assume $\eta:[0,1]\to X$ is a geodesic and set $f=b_{\gamma^{+}}+b_{\gamma^{-}}$.
Convexity implies that $t\mapsto f(\eta(t))$ achieves its maximum
at $0$ or $1$ if it is not constant. So assume $f(x)<0$ for some
$x$. Then there is a geodesic $\eta:[0,1]\to X$ connecting $x$
and $x_{0}=\gamma(0)$. From the assumption we can extend $\eta$
beyond $x_{0}$ such that $\tilde{\eta}:[0,1+\epsilon]\to X$ is a
local geodesics agreeing with $\eta$ on $[0,1]$. Let $a\in[0,1)$
such that $\tilde{\eta}_{|[a,1+\epsilon]}$ is a geodesic. Then 
\[
\max_{t\in[a,1+\epsilon]}f(\tilde{\eta}(t))=f(\eta(1))=0.
\]
However, this implies that $f(\tilde{\eta})$ is constant on $[a,1+\epsilon]$.
But then $f(\eta)$ also attains its maximum at $t=a$ implying that
$f(\eta(t))=0$ for all $t\in[0,1]$. 

The equality shows that we may glue the asymptotic rays. Non-branching
implies that this can be done for at most one pair of rays. Thus for
every $x\in X$ there are unique line $\eta_{x}$ parallel to $\gamma$
such that $x\in\eta_{x}$ and $\eta_{x}(b_{\gamma\text{\textsuperscript{+}}}(x))=x$.
Note also that $\eta_{\eta(t)}=\eta_{\eta(s)}$ for $s,t\in\mathbb{R}$. 

Let $\Xi=\{\eta_{x}\,|\,x\in X\}\subset\operatorname{Lip}(\mathbb{R},X)$.
By local compactness we can show that $\Psi:\Xi\to b_{\gamma^{+}}^{-1}(0)$
defined by $\Psi(\eta_{x})=\eta_{x}(0)$ is a homeomorphism. In particular,
the assignment $x\mapsto(\eta_{x}(0),b_{\gamma^{+}}(x))$ is a homeomorphism
between $X$ and $b_{\gamma^{+}}^{-1}(0)\times\mathbb{R}$.
\end{proof}
\begin{rem*}
Without local extendibility the result might be wrong. Indeed, if
$X$ is the product of a filled triangle and the real line then there
exist two convex functions such that their sum is non-positive and
somewhere negative, and they sum up to zero at the line formed by
a vertex of the triangle. However, if we assume that $(X,d)$ is Busemann
concave then one can use the fact that 
\[
b_{\gamma^{\pm}}(\eta(0))=b_{\eta^{\mp}}(\gamma(0))
\]
to show affinity of the Busemann function. We leave the details to
the interested reader.
\end{rem*}
The following is an analogue of the case of standard Busemann functions.
The result also holds for $\tilde{b}_{x_{0}}$ and $\hat{b}_{x_{0}}$. 
\begin{lem}
\label{lem:Cheeger-Gromoll-ray}Assume $(X,d)$ is locally compact
and uniformly smooth. Then for any $x\in X$ there is a ray $\gamma_{x}:[0,\infty)\to X$
emanating from $x$ such that 
\[
b_{x_{0}}(\gamma_{x}(t))=b_{x_{0}}(x)+t.
\]
\end{lem}
\begin{proof}
This is true for $b_{\gamma_{x_{0}}}$ where $\gamma_{x_{0}}$ is
a ray emanating from $x_{0}$. From the definition there is a sequence
$(\gamma_{x_{0}}^{n})$ of rays emanating from $x_{0}$ such that
$b_{x_{0}}(x)=\lim_{n\to\infty}b_{\gamma_{x_{0}}^{n}}(x)$. Let $\gamma_{x}^{n}$
be the rays emanating from $x$ with $b_{\gamma_{x_{0}}^{n}}(\gamma_{x}^{n}(t))=b_{\gamma_{x_{0}}^{n}}(x)+t$.
By local compactness we can assume $\gamma_{x_{0}}^{n}\to\gamma_{x_{0}}$
and $\gamma_{x}^{n}\to\gamma_{x}$. Thus 
\begin{eqnarray*}
b_{x_{0}}(x)+t & = & \lim_{n\to\infty}b_{\gamma_{x_{0}}^{n}}(x)+t\\
 & = & \lim_{n\to\infty}b_{\gamma_{x_{0}}^{n}}(\gamma_{x}^{n}(t))\\
 & = & b_{\gamma_{x_{0}}}(\gamma_{x}(t))\le b_{x_{0}}(\gamma_{x}(t)).
\end{eqnarray*}
However, $b_{x_{0}}$ is $1$-Lipschitz implying $|b_{x_{0}}(\gamma_{x}(t))-b_{x_{0}}(x)|\le t$
and thus equality above.
\end{proof}
\begin{thm}
Any uniformly smooth locally compact metric space $(X,d)$ admits
a quasi-convex exhaustion function $b:X\to\mathbb{R}$ with compact
sublevels such that $S=b^{-1}(\min b)$ has empty interior.
\end{thm}
\begin{proof}
Let $b=b_{x_{0}}$ for some $x_{0}\in X$. The fact that $S_{x_{0}}$
has empty interior follows from the lemma above. Indeed, let $x\in S_{x_{0}}$.
There is a ray $\gamma_{x}$ emanating from $x$ with $b(\gamma_{x}(t))=b(x)+t$.
Assume $\gamma_{x}(t)\in S_{x_{0}}$ then 
\[
\min b=b(\gamma_{x}(t))=b(x)+t=\min b+t
\]
which can only hold if $t=0$. Therefore, $\interior S_{x}=\varnothing$.
\end{proof}
We call $S_{x_{0}}$ a \emph{weak soul} as there is no way to dissect
it further without an intrinsic notion of boundary (see also below).
\begin{cor}
\label{cor:point-soul}If, in addition, $b$ is strictly quasi-convex
then $S$ is a single point.
\end{cor}
\begin{proof}
Let $x,y\in S_{x_{0}}$ be two point and $\gamma:[0,1]\to S_{x_{0}}$
be a geodesic connecting $x$ and $y$. By strict quasi-convexity
we have $b(\gamma_{t})\le\max\{b(x),b(y)\}$ with strict inequality
if $x\ne y$. But that case cannot happen because $b(x),b(y)=\min b$.
\end{proof}
Strict quasi-convexity of $b$ implies that the sublevel sets are
strictly convex. For Alexandrov spaces we can show the converse.
\begin{prop}
\label{prop:Alex-str-conv}Assume $(M,d)$ is an Alexandrov space
of non-negative curvature. Then $C_{s}$ is strictly convex (for some
$s>\min b$) if and only if $b$ is strictly quasi-convex on $C_{s}$. 
\end{prop}
\begin{rem*}
It is well possible that $b$ is not strictly quasi-convex outside
of $C_{s}$. An example is given as follows: glue the half cylinder
$[0,\infty)\times S^{n-1}$ and the lower hemisphere $S_{\frac{1}{2}}^{n}$
along their boundaries which are isometric to $S^{n-1}$. Then $b$
is up to a constant the distance from the south pole and is not strictly
(quasi-)convex outside of the hemisphere.
\end{rem*}
\begin{proof}
Obviously strict quasi-convexity of $b$ on $C_{s}$ implies strict
convexity of $C_{s}$. Assume instead $C_{s}$ is strictly convex.
We will use the rigidity of the distance from the boundary proven
by \cite{Yamaguchi2012}. From Corollary \ref{cor:CG-fcn-via-bdry}
below we have $(b)_{|C_{s}}=s-b_{C_{s}}$ where $b_{C_{s}}(x)=d(x,\partial C_{s})$.
If $b_{C_{s}}$ was not strictly quasi-convex then there is a non-constant
geodesic $\gamma$ in $C_{s}$ such that $b_{C_{s}}$ is constant
along $\gamma$. The rigidity result in \cite[Proposition 2.1]{Yamaguchi2012}
shows that there is a non-constant geodesic $\eta:[0,1]\to\partial C_{s}$
which is impossible by strict convexity of $C_{s}$. 
\end{proof}
The above actually gives a more general characterization: a closed
convex set $C$ in an Alexandrov space of non-negative curvature is
strictly convex iff $b_{C}(\cdot)=d(\cdot,\partial C)$ is strictly
quasi-convex on $C$. 

\subsection{An application of the technique in the smooth section}

In this section we apply the technique above in the smooth setting.
We show that if a Finsler manifold with non-negative flag curvature
has vanishing tangent curvature along a geodesic $\eta$ then any
Busemann function is convex along $\eta$. This can be used to simplify
the proof of orthogonality of certain tangent vectors in \cite{Lakzian2014}
and avoid a complicated Toponogov-like theorem proved in \cite{Kondo2012}.
In order to avoid a lengthy introduction, we refer the reader to \cite{Ohta2008}
for the notation used in this section. The focus will be on the proof
of uniform smoothness of the distance \cite[Theorem 4.2, Corollary 4.4]{Ohta2008}. 

Note that in the Finsler setting a (forward) geodesic refers to an
constant-speed auto-parallel curve $\gamma:[0,1]\to M$ such that
$d_{F}(\gamma_{0},\gamma_{1})=F(\dot{\gamma}_{0})$ where $d_{F}$
is the asymmetric metric induced by the Finsler structure $F$. Assuming
smoothness of $\gamma:[0,1]\to M$, this is equivalent to 
\[
d_{F}(\gamma_{t},\gamma_{s})=(s-t)d_{F}(\gamma_{0},\gamma_{1})
\]
for $1\ge s\ge t\ge0$.
\begin{lem}
Let $(M,F)$ be a connected forward geodesically complete $C^{\infty}$-Finsler
manifold. Assume $(M,F)$ has non-negative flag curvature and for
all $x\in M$ the norms $F_{x}$ are $2$-uniformly smooth for some
constant $S\ge1$. If $\eta:[0,1]\to M$ is a (forward) geodesic with
$\mathcal{T}=0$ on $T_{\eta}M$ then 
\[
d^{2}(x,\eta_{t})\ge(1-t)d^{2}(x,\eta_{0})+td^{2}(x,\eta_{1})-(1-t)tS^{2}d(\eta_{0},\eta_{1}).
\]
\end{lem}
\begin{proof}
We only sketch the argument as the calculations are exactly those
of \cite[Proof of Theorem 4.2]{Ohta2008}. The only time the assumption
$\mathcal{T}\ge-\delta$ is applied is for estimating 
\[
g_{T(r)}(D_{U}^{T}U,T)=g_{T(r)}(D_{U}^{U}U,T)-\mathcal{T}_{T(r)}(v)
\]
where $T(r)\in T_{\eta_{0}}M$ and $v=\frac{\dot{\eta}_{0}}{F(\dot{\eta}_{0})}$.
Since $D_{U}^{U}U(r)=0$ and $\mathcal{T}=0$ on $T_{\eta_{0}}M$
we see that $g_{T(r)}(D_{U}^{T}U,T)=0$. In particular, if it is possible
to choose $\delta=0$. Then following the calculation we obtain the
result via \cite[Corollary 4.4]{Ohta2008}.
\end{proof}
Along the lines of the proof of Proposition \ref{prop:convex-Busemann}
we immediately deduce.
\begin{cor}
Let $(M,F)$ be as above. Then any Busemann function associated to
a ray $\gamma:[0,\infty)\to M$ is convex along $\eta$. 
\end{cor}
The next result was proved in \cite[Lemma 4.8]{Lakzian2014} for closed
forward geodesics, i.e. a map $\eta:S^{1}\to M$ such that $\eta$
is a locally geodesic. Note that we do not need the reversibility
assumption of the closed geodesic \cite[Theorem 1.2]{Lakzian2014}.
The author wonders whether $\mathcal{T}=0$ on $T_{\eta}M$ would
imply that the reversed $\tilde{\eta}:t\mapsto\eta(1-t)$ is a geodesic
loop as well. This would be the case if $\mathcal{T}=0$ in a neighborhood
$U$ of $\eta$, i.e. $M$ is Berwaldian in $U$.
\begin{cor}
Let $(M,F)$ be as above and assume $\eta:[0,1]\to M$ is a (forward)
geodesic loop, i.e. $\eta$ is locally (forward) geodesic such that
$\eta_{0}=\eta_{1}$. Then for any ray $\gamma:[0,\infty)\to M$ with
$\gamma_{0}=\eta_{0}$ it holds 
\[
g_{\dot{\gamma}_{0}}(\dot{\gamma}_{0},\dot{\eta}_{0})\ge0\quad\mbox{and}\quad g_{\dot{\gamma}_{0}}(\dot{\gamma}_{0},\dot{\eta}_{1})\le0.
\]
In particular, if $\eta$ is a (forward) closed geodesic then $g_{\dot{\gamma}_{0}}(\dot{\gamma}_{0},\dot{\eta}_{0})=0$. 
\end{cor}
\begin{proof}
This is a direct consequence of the first variation formula and the
convexity of the Busemann function. Indeed, by the first variation
formula (see \cite[Exercise 5.2.4]{BCS2000}) and uniqueness of geodesics
between $\gamma_{0}$ and $\gamma_{t}$ we have 
\[
\lim_{s\to0^{+}}\frac{d(\gamma_{t},\eta_{s})-d(\gamma_{t},\eta_{0})}{d(\eta_{0},\eta_{s})}=g_{\dot{\gamma}_{0}}(\dot{\gamma}_{0},\frac{\dot{\eta}_{0}}{F(\dot{\eta}_{0})})
\]
and 
\[
\lim_{s\to1^{-}}\frac{d(\gamma_{t},\eta_{s})-d(\gamma_{t},\eta_{0})}{d(\eta_{s},\eta_{1})}=g_{\dot{\gamma}_{0}}(\dot{\gamma}_{0},-\frac{\dot{\eta}_{1}}{F(\dot{\eta}_{1})}).
\]

Now convexity of the Busemann function $b_{\gamma}$ associated to
$\gamma$ implies that 
\[
b_{\gamma}(\eta_{s})\le b_{\gamma_{0}}(\eta_{0})=0.
\]
However, we have 
\[
d(\gamma_{t},\eta_{0})-d(\gamma_{t},\eta_{s})=t-d(\gamma_{t},\eta_{s})\le b_{\eta}(\gamma_{s})
\]
which immediately shows that $g_{\dot{\gamma}_{0}}(\dot{\gamma}_{0},\dot{\eta}_{0})\ge0$.
A similar argument shows $g_{\dot{\gamma}_{0}}(\dot{\gamma}_{0},-\dot{\eta}_{1})\ge0$. 

If $\gamma$ is a closed geodesic then also $\dot{\eta}_{0}=\dot{\eta}_{1}$
so that 
\[
0\le g_{\dot{\gamma}_{0}}(\dot{\gamma}_{0},\dot{\eta}_{0})=g_{\dot{\gamma}_{0}}(\dot{\gamma}_{0},\dot{\eta}_{1})\le0.
\]
\end{proof}

\subsection{A strong deformation retract onto a weak soul }

Having a weak soul shows that all geodesic loops starting in $S_{x_{0}}$
must stay in $S_{x_{0}}$. Thus if all loops starting at some $x\in S_{x_{0}}$
are homotopic to a geodesic loop starting at $x$ then the fundamental
group of $X$ and $S_{x_{0}}$ are the same. In general, it is not
clear how to use a similar argument for higher homotopy groups. If
$S_{x_{0}}$ was a strong deformation retract of $X$ then the all
fundamental group would agree. 

Throughout this section we assume that $(X,d)$ is a locally compact
and uniformly smooth. Note that this implies that closed bounded sets
are compact. We want to give a condition which implies that each sublevel
of $b_{x_{0}}$ is a strong deformation retract of $X$. 

The first result is just reformulation of \cite[Proposition 1.3]{Cheeger1972}.
The only ingredient used in their proof is the fact that the sublevels
of $b_{x_{0}}$ are closed and totally geodesic.
\begin{lem}
\label{lem:bdry-dist}The sublevel sets $C_{s}=b_{x_{0}}^{-1}([0,s])$
are compact and the boundaries have the following form: Let $0<s<t$
then 
\[
\partial C_{s}=\{x\in C_{t}\,|\,d(x,\partial C_{t})=t-s\}.
\]
\end{lem}
\begin{proof}
Note by definition $b_{x_{0}}(\gamma(t))=t$ for any ray starting
at $x_{0}$. In particular, $b_{x_{0}}$ is unbounded when restricted
to such a ray. Now suppose that sublevels of $b_{x_{0}}$ are not
compact. Then there is an $s$ and a sequence $x_{n}\to\infty$ with
$b_{x_{0}}(x_{n})\le s$. Let $\gamma_{n}$ be a geodesic connecting
$x_{0}$ and $x_{n}$. By local compactness there is a subsequence
of $(\gamma_{n})$ converging to a ray $\gamma$ starting at $x_{0}$.
As $C_{s}=b_{x_{0}}^{-1}([0,s])$ is closed convex, $\gamma_{n}$
and $\gamma$ lie entirely in $C_{s}$. This, however, implies that
$b_{x_{0}}(\gamma(t))\le s$ which is a contradiction.

To see that second claim, note $b_{x_{0}}$ is $1$-Lipschitz so that
$x\in\partial C_{s}$ and $y\in C_{t}$ with $d(x,y)<t-s$ implies
$b_{x_{0}}(y)<t$.
\end{proof}
The following shows that $b_{x_{0}}$ can be described in a local
manner. It can be used to show that a certain notion of positive curvature
means that the weak soul $S_{x_{0}}$ consists only of a single point,
see Proposition \ref{prop:pos-curv} below. 
\begin{cor}
\label{cor:CG-fcn-via-bdry}For each $r>m$ where $m=\min b_{x_{0}}$
define the function $b_{C_{r}}:C_{r}\to[0,r-m]$ by 
\[
b_{C_{r}}(x)=d(x,\partial C_{r}).
\]
Then $b_{C_{r}}$ is quasi-concave and it holds 
\[
b_{x_{0}}(x)=r-b_{C_{r}}(x)
\]
for $x\in C_{r}$.
\end{cor}
\begin{proof}
If $x\in\partial C_{s}$ for $s\le r$ then $b_{x_{0}}(x)=s$. Thus
the lemma implies 
\[
b_{x_{0}}(x)=r-(r-s)=r-b_{C_{r}}(x).
\]
It remains to show that $b_{x_{0}}^{-1}(s)=\partial C_{s}$. Let $x\in b_{x_{0}}^{-1}(s)$.
By Lemma \ref{lem:Cheeger-Gromoll-ray} there is a ray $\gamma_{x}$
emanating from $x$ such that 
\[
b_{x_{0}}(\gamma_{x}(t))=b_{x_{0}}(x)+t.
\]
Since $\gamma_{x}(r-s)\in\partial C_{r}$ we obtain $d(x,\partial C_{r})\le r-s$.
As $b_{x_{0}}$ is $1$-Lipschitz this must be an equality so that
the previous lemma implies $x\in\partial C_{s}$.
\end{proof}
\begin{cor}
\label{cor:subspace-contraction}Assume $(X,d)$ is non-branching
and let $A_{s}$ be the set of points $x\in\partial C_{s}$ such that
$d(x,S_{x_{0}})=s-\min b_{x_{0}}$. Then $A_{r,s}=\cup_{r\le r'\le s}A_{r'}$
is homeomorphic to $A_{s}\times[0,1]$ for all $r,s>\min b_{x_{0}}$.
Furthermore, there is a continuous map $\Phi_{r}:A_{r}\to S_{x_{0}}$.
\end{cor}
\begin{proof}
We only indicate the proof as the construction is similar to the one
above: The ray $\gamma$ given by Lemma \ref{lem:Cheeger-Gromoll-ray}
can be extended backwards to a ray starting in $S_{x_{0}}$. More
precisely, if $\gamma'$ is any unit speed geodesic $\gamma'$ connecting
$y_{x}$ and $x$ where $y_{x}$ is a foot point of $x$ in $S_{x_{0}}$.
Because $d(x,S_{x_{0}})=b_{x_{0}}(x)-\min b_{x_{0}}$, $\gamma\cup\gamma'$
is a ray emanating from $y_{x}$. By non-branching and compactness
of $S_{x_{0}}$ we see that there is exactly one $y\in S_{x_{0}}$
with $d(x,y)=d(x,S_{x_{0}})$. Furthermore, the assignment $x\mapsto y_{x}$
is continuous. 

Furthermore, note that $\gamma'$ intersects each $A_{r}$ with $\min b_{x_{0}}<r\le s$
in exactly one point. This shows that $A_{r,s}$ is homeomorphic to
$A_{s}\times[0,1]$.
\end{proof}
In general it is not true that $A_{s}=\partial C_{s}$. Indeed, assume
$X$ is a space which is obtained by glueing half a disk of diameter
$1$ with a flat half-strip of width $1$. Then it is easy to see
that $\cup_{s\ge\min b_{x_{0}}}A_{s}$ corresponds to the ray starting
at the ``south pole'' of the disk. 

The following result shows that for a subclass with nice gradient
flow behavior there is a contractive deformation retract. Note, however,
it is expected that only Riemannian-like metric space have such property
(see \cite{OS2012}). For a general overview of gradient flows on
metric spaces we refer to \cite{AmbGigSav2008}. In the setting of
Alexandrov spaces the contractive behavior was shown in \cite{Perelman1994,Lytchak2006}.
\begin{prop}
Assume $(X,d)$ is $p$-uniformly smooth and gradient flows of convex
functions exist and are contractive. Then there is a strong deformation
retract $F:X\times[0,1]\to X$ onto $S_{x_{0}}$ such that $F_{t}$
is a contraction onto a sublevel set of $b_{x_{0}}$. In particular,
$F_{1}:X\to S_{x_{0}}$ is a contraction.
\end{prop}
\begin{rem*}
Note that only convexity of $b_{x_{0}}$ is needed. More generally,
a slight adjustment of proof works if $b_{x_{0}}$ is $p$-convex.
It is unclear whether quasi-convexity of $b_{x_{0}}$ would be sufficient.
\end{rem*}
\begin{proof}
If $b_{x_{0}}$ is convex so is $b_{x_{0}}^{r}=\max\{b_{x_{0}},r\}$.
Denote the gradient flows of $b_{x_{0}}^{r}$ by $\Phi_{t}^{r}:X\to X$.
As the definition is local we see that for $x\in X$ with $b_{x_{0}}(x)>r\ge r'$
there is a $t_{x}^{r}\in(0,\infty)$ such that $\Phi_{t}^{r}(x)=\Phi_{t}^{r'}(x)$
for $t\in[0,t_{x}^{r}]$. More precisely, the equality fails once
$b_{x_{0}}(\Phi_{t}^{r}(x))\le r$. Similarly, one can show that $\Phi_{t}^{r}$
restricted to $C_{r}$ is the identify. Note that any $r\ge\min b_{x_{0}}$
is reached in finite time. Furthermore, $\Phi_{t}^{r}$ is constant
on $C_{r}$. 

For any $x,y\in X$ and $t\ge\max\{t_{x}^{r},t_{y}^{r}\}$ it holds
\begin{eqnarray*}
\Phi_{t}^{r}(x) & = & \Phi_{t_{x}}^{r}(x)\\
\Phi_{t}^{r}(y) & = & \Phi_{t_{y}}^{r}(y).
\end{eqnarray*}
The contraction property implies 
\[
d(\Phi_{t}^{r}(x),\Phi_{t}^{r}(y))\le d(x,y).
\]
Define now 
\[
F_{r}(x)=\lim_{t\to\infty}\Phi_{t}^{r}(x).
\]
Since each $x$ reaches $C_{r}$ in finite time, the map is well-defined.
Furthermore, we see that $F_{r}$ is contractive and maps $X$ onto
$C_{r}$. Also note that $F_{r}$ is the identify on $C_{r}$. 

If $\phi:[0,1)\to[\min b_{x_{0}},\infty)$ then the following map
$F:X\times[0,1]\to X$ satisfies the required properties: 
\[
F(x,s)=\begin{cases}
F_{\phi(1-s)}(x) & s>0\\
x & s=0.
\end{cases}
\]
We only need to show that $(x_{n},s_{n})\to(x,0)$ implies $F(x_{n},s_{n})\to x$.
However, this follows because there is an $r>0$ such that $x_{n},x\in C_{r}$.
Since $t_{n}=\phi(1-s_{n})\to\infty$ we see that for large $n$ it
holds $t_{n}\ge r$ and thus 
\[
F(x_{n},s_{n})=F_{t_{n}}(x_{n})=x_{n}\to x.
\]
\end{proof}
For non-Riemannian-like metric spaces the above construction does
not work. However, if we assume that locally projections onto convex
sets are unique then it is still possible to construction a strong
deformation retract. Such an assumption would imply a local weak upper
curvature bound which is not satisfied on Alexandrov spaces. Indeed,
Petrunin gave an example of an Alexandrov space which contains a convex
set without such a property \cite{Petrunin2013}. As this is not published
anywhere else here a short construction: Let $X$ be the doubling
of the (convex) region $\{(x,y)\in\mathbb{R}^{2}\,|\,y\ge x^{2}\}$.
Now let $A$ be the (doubling) of $F$ intersected with the ball with
center $(0,-1)$ passing through $(1,1)$. This set is convex in $X$.
Every point on the boundary of $F$ outside of $A$ has two projection
points onto $A$. 
\begin{defn}
[strict convexity radius] The (strict) convexity radius $\rho(x)$
of a point $x$ in a metric space $(X,d)$ is the supremum of all
$r\ge0$ such that the closed ball $\bar{B}_{r}(x)$ is (strictly)
convex. 
\end{defn}

\begin{defn}
[injectivity radius] The injectivity radius $i(x)$ of a point $x\in X$
is the supremum of all $r\ge0$ such that for all $y\in B_{r}(x)$
there is only one geodesic connecting $x$ and $y$.
\end{defn}
\begin{lem}
\label{lem:str-convex-str-inj}Assume $(X,d)$ has strict convexity
radius locally bounded away from zero. Then the injectivity radius
is locally bounded away from zero.
\end{lem}
\begin{rem*}
The lemma shows that if $(X,d)$ was in addition also a weak Busemann
$G$-space then it is actually a Busemann $G$-space (see below). 
\end{rem*}
\begin{proof}
Choose some large ball $B_{R}(x_{0})$ such that $\rho(x)\ge\epsilon$
for all $x\in B_{R}(x_{0})$ and $\epsilon\le\frac{R}{4}$. We claim
that $i(x)\ge2\epsilon$. Indeed, let $y\in B_{2\epsilon}(x)$ for
$x\in B_{\frac{R}{4}}(x_{0})$ and set $2d=d(x,y)<2\epsilon$. Then
$B_{d}(x)\cap B_{d}(y)=\varnothing$, and $\bar{B}_{d}(x)\cap\bar{B}_{d}(y)$
is non-empty and consists entirely of midpoints of $x$ and $y$.
However, if there were two distinct points $m,m'\in\bar{B}_{d}(x)\cap\bar{B}_{d}(y)$
then their midpoint would be in $B_{d}(x)\cap B_{d}(y)$ which is
a contradiction. This implies that midpoints in $B_{R}(x_{0})$ are
unique. Local compactness immediately gives uniqueness of the geodesics.
\end{proof}
\begin{lem}
Assume $(X,d)$ has strict convexity radius locally bounded away from
zero. Let $C$ be a compact convex subset of $X$ then there is an
$\epsilon>0$ such that any point $y$ in the $\epsilon$-neighborhood
$C_{\epsilon}$ of $C$ has a unique point $y_{C}\in C$ such that
$d(y,y_{C})=d(y,C)$.
\end{lem}
\begin{proof}
For every $x_{0}\in X$ and $R>0$ there is an $\epsilon>0$ such
that $\rho(x),i(x)\ge2\epsilon$ for all $x\in B_{R}(x_{0})$. Assume
$C\subset B_{\frac{R}{2}}(x_{0})$ and $\epsilon<\frac{R}{2}$. We
claim that $C_{\epsilon}$ has the required property. 

Let $y\in C_{\epsilon}$. Since $C$ is compact, there is at least
one such $y_{C}$. Assume there are two $y_{C},y_{C}^{'}\in C$ with
$d(y,y_{C})=d(y,y_{C}^{'})=d(y,C)\le\epsilon$. From the definition
we have $y_{C},y_{C}^{'}\in\partial B_{\epsilon}(y)$ with $\interior B_{\epsilon}(y)\cap C\ne\varnothing$.
Assume $m$ is the midpoint of $y_{C}$ and $y_{C}^{'}$ then by convexity
of $C$, $m\in C$. However, this implies that $m\notin\interior B_{\epsilon}(x)$.
Then strict convexity of $B_{\epsilon}(y)$ shows $y_{C}=y_{C}^{'}$,
i.e. $y$ has exactly one closest point in $C$.
\end{proof}
\begin{cor}
Assume there is a convex set $C'$ with $C\subset C'\subset C_{\epsilon}$.
Then the projection induces a strong deformation retract of $C'$
onto $C$.
\end{cor}
\begin{proof}
Denote the closest point projection by $p:C'\to C.$ As the geodesic
$\gamma:[0,1]\to X$ connecting $x$ and $p(x)$ is unique we may
define a map $P:C'\times[0,1]\to C'$ by 
\[
P(x,t)=\gamma_{t}.
\]
Note that $P(x,0)=x$, $P(x,1)=p(x)$ and by convexity of $C'$ we
also have $P(x,t)\in C'$ for all $(x,t)\in C'\times[0,1]$. Furthermore,
local compactness together with local uniqueness of the geodesics
shows that $P$ is continuous and thus a strong deformation retract.
\end{proof}
\begin{thm}
Assume $(X,d)$ has strict convexity radius locally bounded away from
zero and there a quasi-convex continuous function $b:X\to[0,\infty)$
with compact sublevels. Then for any $r>0$ the sets $C_{r}=b^{-1}([0,r])$
are strong deformation retracts of $(X,d)$. In particular, $X$ retracts
onto the weak soul $S$.
\end{thm}
\begin{proof}
The construction is similar to \cite[Section 2]{Cheeger1972}. By
the above lemma any point $x\in C_{s}$ admits a unique projection
point $f_{s,r}^{t}(x)\in C_{(1-t)s+tr}$ if $s\ge r$ and $|s-r|\le\epsilon(b,s)$
where $\epsilon$ is given as above by choosing $R$ such that $C_{s}\subset B_{\frac{R}{2}}(x_{0})$.
Note that $s\mapsto\epsilon(b,s)$ is uniformly bounded away from
zero on $[0,M]$ for all $M>0$.

Now define $s_{0}=r$ and $s_{n+1}=s_{n}+\frac{1}{2}\epsilon(b,s_{n})$.
We claim that $s_{n}\to\infty$ as $n\to\infty$. Indeed, if $s_{n}\to s<\infty$
then $\epsilon(b,s_{n})\to0$ which is a contradiction because $\epsilon(b,s_{n})\ge\epsilon(b,M)$
whenever $s_{n}\le M$. 

Let $t\in[r,\infty)$. We define a family of functions $(f_{n}^{t}:C_{s_{n}}\to C_{s_{n}})_{n\in\mathbb{N}}$
as follows: Set $f_{0}^{t}=\operatorname{id}$ and 
\[
f_{n+1}^{t}=\begin{cases}
f_{n}^{t}\circ f_{s_{n},s_{n-1}} & t\le s_{n-1}\\
f_{s_{n},s_{n-1}}^{\lambda} & t\in[s_{n-1},s_{n}]\:\mbox{ with }\lambda=\frac{t-s_{n-1}}{s_{n}-s_{n-1}}\\
\operatorname{id} & t\ge s_{n}.
\end{cases}
\]
One can verify that $f_{n+1}^{t}$ is continuous. Furthermore, let
$m\le n$ then 
\[
f_{n}^{t}=f_{m}^{t}\quad\mbox{on }C_{s_{m}},t\ge s_{m}
\]
and 
\[
f_{n}^{t}(x)=x\qquad\mbox{if }(x,t)\in C_{s_{m}}\times[s_{m},\infty).
\]

Now let $\phi:[0,1)\to[r,\infty)$ be an increasing homeomorphism.
Then define 
\[
F(x,t)=\begin{cases}
\lim_{n\to\infty}f_{n}^{\phi(t)}(x) & t<1\\
x & t=1.
\end{cases}
\]
Note that $F$ is continuous on $X\times[0,1)$. So let $(x_{n},t_{n})\to(x,1)$
then $x_{n},x\in C_{s_{m}}$ for some large $m$. As $t_{n}\to1$
we have $s_{m}\le\phi(t_{n})$ for $n\ge n_{0}$ and thus 
\[
F(x_{n},t_{n})=f_{m}^{\phi(t_{n})}(x_{n})=x_{n}\to x=f_{m}^{\phi(1)}(x)=F(x,1).
\]
 This shows that $F$ is a strong deformation retract because $f_{n}^{0}(C_{s_{n}})=C_{r}.$
\end{proof}

\subsection{Towards a soul on uniformly smooth spaces}

To complete the soul theorem alone the lines of Cheeger-Gromoll one
needs a proper notion of intrinsic boundary. We say that geodesics
in $(X,d)$ can be \emph{locally extended} if for all geodesics $\gamma:[0,1]\to X$
there is a local geodesic $\tilde{\gamma}:[-\epsilon,1]\to X$ agreeing
on $[0,1]$ with $\gamma$. In the smooth setting, $x\in\partial M$
would imply that there is a geodesic starting at $x$ such that $-\dot{\gamma}_{x}\ne T_{x}M$.
In particular, $\gamma$ cannot be extended beyond $x$. 
\begin{defn}
[geodesic boundary] The \emph{geodesic boundary} $\partial_{g}C$
of subset $C\subset X$ is the set of all $x\in C$ such that there
is a geodesic in $C$ that cannot be locally extended in $C$.
\end{defn}
\begin{rem*}
The notion differs from the boundary of Alexandrov spaces. Indeed,
if the singular set is dense then the geodesic boundary is dense as
well. 
\end{rem*}
The following is a replacement for the notion of closed manifold. 
\begin{defn}
[weak Busemann $G$-space] A geodesic metric space is a weak Busemann
$G$-space if its geodesics can be locally extended in a unique way.
In particular, $\partial_{g}X=\varnothing.$ 
\end{defn}
Unique extendability implies non-branching. The Heisenberg group equipped
with a left-invariant Carnot-Caratheodory metric is a weak Busemann
$G$-space which is not a (strong) Busemann $G$-space. See \cite{Busemann1955}
for more on Busemann $G$-spaces. From Lemma \ref{lem:str-convex-str-inj}
we immediately see that.
\begin{lem}
\label{lem:weak-Busemann-to-strong}A weak Busemann $G$-space with
strict convexity radius locally bounded away from zero is a Busemann
$G$-space.
\end{lem}
By \cite[Theorem 1.6]{Cheeger1972} the geodesic boundary above agrees
with the boundary of open convex subsets if $(X,d)$ is a Riemannian
manifold. Their proof also works for Finsler manifolds. In the current
setting this is almost true. 
\begin{lem}
If $(X,d)$ is a weak Busemann $G$-space and $C=\closure(\interior C)$
is convex then 
\[
\partial C=\operatorname{cl}\partial_{g}C.
\]
\end{lem}
\begin{proof}
The inclusion $\partial_{g}C\subset\partial C$ follows from local
extendability and the fact that $C=\closure(\interior C)$. 

Let $x\in\partial C$ then there are $y_{n}\in\interior C$ and $z_{n}\in X\backslash C$
such that $y_{n},z_{n}\to x$. Furthermore, there is a geodesic $\gamma_{n}:[0,1]\to X$
connecting $y_{n}$ and $z_{n}$. Observe that for any $\epsilon>0$
and sufficiently large $n$ it holds $\gamma_{n}(t)\in B_{\epsilon}(x)$. 

By convexity of $C$ there is a $t_{0}>0$ such that $\gamma_{n}([0,t_{0}])\subset C$
and $\gamma_{n}((t_{0},1])\subset(X\backslash C)$. But this means
$\gamma_{n}(t_{0})\in\partial_{g}C$ proving the claim.
\end{proof}
\begin{defn}
[Gromov non-negative curvature] A weak Busemann $G$-space is said
to be \emph{non-negatively curved in the sense of Gromov} if for all
closed convex sets $C$ the functions 
\[
b_{C}:x\mapsto d(x,\partial_{g}C)
\]
are quasi-concave on $C$, i.e. the superlevels of $b_{C}:C\to[0,\infty)$
are convex. If $b_{C}$ is strictly quasi-concave at $x_{0}$, then
we say the space has \emph{positive curvature in the sense of Gromov
at $x_{0}$}. If it holds for all $x_{0}\in X$ then we just say $(X,d)$
has positive curvature in the sense of Gromov. Furthermore, we say
it has \emph{strong non-negative (positive) curvature} in the sense
of Gromov if $b_{C}$ is convex (strictly convex).
\end{defn}
\begin{rem*}
(1) In Gromov's terminology \cite[p.44]{Gromov1991} this means that
the inward equidistant sets $(\partial C)_{-\epsilon}$ of the convex
hypersurface $\partial C$ remain convex.

(2) If $C=\closure(\interior C)$ then $b_{C}(x)=d(x,\partial C)$
by the lemma above.

(3) The opposite curvature bound is usually called Pedersen convex,
resp. Pedersen non-positive/negative curvature \cite[(36.15)]{Pedersen1952,Busemann1955}.
This property says that the $r$-neighborhood of convex sets remain
convex. It is sometimes called ``has convex capsules''. Note that
this characterization was rediscovered by Gromov \cite[p.44]{Gromov1991}
as ``the outward equidistant sets $(\partial C)_{\epsilon}$ of a
convex convex hypersurface $\partial C$ remain convex''.
\end{rem*}
Note that Gromov non-negative curvature is rather weak. It is trivially
satisfied on weak Busemann $G$-spaces whose only closed convex subset
with non-trivial boundary are geodesic. Indeed, an example is given
by Heisenberg group, see \cite{Monti2005}.

The author wonders if it is possible to define non-negative curvature
in the sense of Gromov only in terms of (local/global) properties
of the metric not relying on sets.

Assume in the following that $(X,d)$ is proper, uniformly smooth
and is non-negatively curved in the sense of Gromov. Furthermore,
we assume $x_{0}\in X$ is fixed and $C_{r}=b_{x_{0}}^{-1}((-\infty,r])$
are the sublevels of the Cheeger-Gromoll function (see above). 
\begin{prop}
\label{prop:pos-curv}If $(X,d)$ has positive curvature in the sense
of Gromov then $S_{x_{0}}$ is a point.
\end{prop}
\begin{proof}
Let $s>\min b_{x_{0}}$. Then Corollary \ref{cor:CG-fcn-via-bdry}
shows that $b_{x_{0}}$ and $s-b_{C_{s}}$ agree on $C_{s}$. Positive
curvature shows that $b_{C_{s}}$ is strictly quasi-concave so by
Corollary \ref{cor:point-soul} the set $S_{x_{0}}$ must be a point.
\end{proof}
Now we want to show that one can successively reduce $S_{x_{0}}$.
For such a reduction we need to assume that the geodesic boundary
does not behave too badly. For this recall that a set $A$ is nowhere
dense in a closed subset $B$ if its closure in $B$ has empty interior
w.r.t. $B$.
\begin{defn}
[non-trivial boundary] A metric space is said to have non-trivial
boundary property if for all non-trivial closed convex set $C$ the
geodesic boundary $\partial_{g}C$ is nowhere dense in $C$.
\end{defn}
By \cite[Theorem 1.6]{Cheeger1972} any Riemannian manifold satisfies
the non-trivial boundary property. 
\begin{lem}
Assume $(X,d)$ has non-trivial boundary property and non-negative
curvature in the sense of Gromov. Then any closed convex set $C$
the set $S=b_{C}^{-1}(\max b_{C})$ is a closed convex subset without
interior w.r.t. $C$. We call $S$ the weak soul of $C$.
\end{lem}
If the Cheeger-Gromoll function $b_{x_{0}}$ is seen as minus of the
renormalized distance from the boundary at infinity then uniform smoothness
implies Gromov non-negative curvature in the (very) large. Corollary
\ref{cor:CG-fcn-via-bdry} shows that the weak soul of $S_{x_{0}}$
actually agrees with the weak soul of the sublevel set thus justifying
the terminology.

Given a convex exhaustion function $b$, one can now start with $X$
and obtain a weak soul $C_{0}=S_{x_{0}}$. This set admits a weak
soul $C_{1}$ as well. Repeatedly applying the lemma shows that there
is a flag of closed convex set $C_{0}\supset C_{1}\supset\cdots$.
In order to show that this procedure eventually ends we need the following.
\begin{defn}
[topological dimension] A metric space $(X,d)$ is said to have topological
dimension $n$, denoted by $\dim_{\operatorname{top}}X=n$, if every
open cover $(U_{\alpha})_{\alpha\in I}$ of $X$ admits a cover $(V_{\beta})_{\beta\in J}$
such that for all $\beta\in J$ there is an $\alpha\in I$ with	 $V_{\beta}\subset U_{\alpha}$
and each $x\in X$ is contained in at most $n$ sets $V_{\beta}$.
\end{defn}
\begin{lem}
Let $C$ be convex and $S$ its weak soul. Then
\[
1+\dim_{\operatorname{top}}S\le\dim_{\operatorname{top}}C.
\]
If $S$ contains a geodesic then $\dim_{\operatorname{top}}S>0$.
\end{lem}
\begin{proof}
This follows along the lines of Corollary \ref{cor:subspace-contraction}.
Indeed, look at the sets $A_{r}\subset\partial C_{r}$ such that 
\[
d(x,S)=r\quad\mbox{ for all }x\in A_{r}
\]
where $C_{r}=b_{C}^{-1}([r,\infty))$.

Then as in Corollary \ref{cor:subspace-contraction} one can show
that sets $A_{r}\times[0,1]$ is homeomorphic to a subset of $C$.
In particular, it holds 
\begin{eqnarray*}
1+\dim_{\operatorname{top}}A_{r} & = & \dim_{\operatorname{top}}A_{r}\times[0,1]\\
 & \le & \dim_{\operatorname{top}}C.
\end{eqnarray*}
Also observe that there is a continuous map of $A_{r}$ onto $S$
so that $\dim_{\operatorname{top}}A_{r}\ge\dim_{\operatorname{top}}S$.

The last statement follows as an embedded line has topological dimension
$1$.
\end{proof}
\begin{cor}
Any finite dimensional, uniformly smooth metric space $(X,d)$ of
non-negative curvature in the sense of Gromov and non-trivial boundary
property admits closed convex set $S$ such that $\partial_{g}S=\varnothing$.
We call $S$ a soul of $X$.
\end{cor}
Note that the construction of the deformation retracts also works
inside of convex sets when replacing the Cheeger-Gromoll exhaustion
$b_{x_{0}}$ by $b_{C}$. Thus we may summarize the results above
as follows. 
\begin{thm}
Assume $(X,d)$ is uniformly smooth, strong non-negatively curved
in the sense of Gromov and has strict convexity radius locally bounded
away from zero. Then there is a strong deformation retract onto a
closed convex set $S$ with $\partial_{g}S=\varnothing$.
\end{thm}
Note that by Lemma \ref{lem:weak-Busemann-to-strong} and \cite{Berestovskii1977}
the assumptions imply that $(X,d)$ is finite dimensional.

\bibliographystyle{amsalpha}
\bibliography{bib}

\providecommand{\bysame}{\leavevmode\hbox to3em{\hrulefill}\thinspace}
\providecommand{\MR}{\relax\ifhmode\unskip\space\fi MR }
\providecommand{\MRhref}[2]{%
  \href{http://www.ams.org/mathscinet-getitem?mr=#1}{#2}
}
\providecommand{\href}[2]{#2}
\begin{thebibliography}{HKST15}

\bibitem[AGS08]{AmbGigSav2008}
L.~Ambrosio, N.~Gigli, and G.~Savar{\'{e}}, \emph{{Gradient flows in metric
  spaces and in the space of probability measures}}, second ed., Lectures in
  Mathematics ETH Z{\"{u}}rich, Birkh{\"{a}}user Verlag, Basel, 2008.

\bibitem[AKP]{Alexander}
S.~Alexander, V.~Kapovich, and A.~Petrunin, \emph{{Alexandrov geometry}},
  https://www.math.psu.edu/petrunin/papers/alexandrov-geometry/.

\bibitem[BB11]{Bjorn2011}
A.~Bj{\"{o}}rn and J.~Bj{\"{o}}rn, \emph{{Nonlinear Potential Theory on Metric
  Spaces}}, European Mathematical Society Publishing House, Zuerich,
  Switzerland, 2011.

\bibitem[BCS00]{BCS2000}
D.~Bao, S.~Chern, and Z.~Shen, \emph{{An Introduction to Riemann-Finsler
  Geometry}}, Springer, 2000.

\bibitem[Ber77]{Berestovskii1977}
V.~N. Berestovskii, \emph{{The finite-dimensionality problem for Busemann
  G-spaces}}, Siberian Mathematical Journal \textbf{18} (1977), no.~1,
  159--161.

\bibitem[BGP92]{Burago1992}
Y.~Burago, M.~Gromov, and G.~Perelman, \emph{{A.D. Alexandrov spaces with
  curvature bounded below}}, Russian Mathematical Surveys \textbf{47} (1992),
  no.~2, 1--58 (en).

\bibitem[BH99]{Bridson1999}
M.~R. Bridson and A.~H{\"{a}}fliger, \emph{{Metric Spaces of Non-Positive
  Curvature}}, Springer, 1999.

\bibitem[BP83]{Busemann1983}
H.~Busemann and B.~B. Phadke, \emph{{Peakless and monotone functions on
  G-spaces}}, Tsukuba Journal of Mathematics \textbf{7} (1983), no.~1,
  105--135.

\bibitem[Bus55]{Busemann1955}
H.~Busemann, \emph{{The geometry of geodesics}}, Academic Press, 1955.

\bibitem[CG72]{Cheeger1972}
J.~Cheeger and D.~Gromoll, \emph{{On the Structure of Complete Manifolds of
  Nonnegative Curvature}}, Ann. of Math. (2) \textbf{96} (1972), no.~3,
  413--443.

\bibitem[DL14]{Descombes2014}
D.~Descombes and U.~Lang, \emph{{Convex geodesic bicombings and
  hyperbolicity}}, Geometriae Dedicata \textbf{177} (2014), no.~1, 367--384.

\bibitem[FLS07]{Foertsch2010}
T.~Foertsch, A.~Lytchak, and V.~Schroeder, \emph{{Nonpositive Curvature and the
  Ptolemy Inequality}}, International Mathematics Research Notices
  \textbf{2007} (2007), 1--15.

\bibitem[Gig13]{Gigli2013a}
N.~Gigli, \emph{{The splitting theorem in non-smooth context}}, arxiv:1302.5555
  (2013).

\bibitem[GKO13]{Gigli2013}
N.~Gigli, K.~Kuwada, and S.~Ohta, \emph{{Heat Flow on Alexandrov Spaces}},
  Communications on Pure and Applied Mathematics \textbf{66} (2013), no.~3,
  307--331.

\bibitem[Gro91]{Gromov1991}
M.~Gromov, \emph{{Sign and geometric meaning of curvature}}, Rendiconti del
  Seminario Matematico e Fisico di Milano \textbf{61} (1991), no.~1, 9--123.

\bibitem[HKST15]{Heinonen2015}
J.~Heinonen, P.~Koskela, N.~Shanmugalingam, and J.~T. Tyson, \emph{{Sobolev
  Spaces on Metric Measure Spaces: An Approach Based on Upper Gradients}},
  Cambridge University Press, 2015.

\bibitem[Hua10]{Hua2010}
B.~Hua, \emph{{Harmonic functions of polynomial growth on singular spaces with
  nonnegative Ricci curvature}}, Proceedings of the American Mathematical
  Society \textbf{139} (2010), no.~6, 2191--2205.

\bibitem[Kan61]{Kann1961}
E.~Kann, \emph{{Bonnet's theorem in two-dimensional G-spaces}}, Communications
  on Pure and Applied Mathematics \textbf{14} (1961), no.~4, 765--784.

\bibitem[Kel14]{Kell2014}
M.~Kell, \emph{{Uniformly convex metric spaces}}, Analysis and Geometry in
  Metric Spaces \textbf{2} (2014), no.~1, 359--380.

\bibitem[Kel15]{Kell2015}
\bysame, \emph{{A note on non-negatively curved Berwald spaces}},
  arxiv:1502.03764 (2015).

\bibitem[Ket15]{Ketterer2015c}
Ch. Ketterer, \emph{Cones over metric measure spaces and the maximal diameter
  theorem}, Journal de Math{\'e}matiques Pures et Appliqu{\'e}es \textbf{103}
  (2015), no.~5, 1228--1275.

\bibitem[KK06]{Kristaly2006}
A.~Krist{\'{a}}ly and L.~Kozma, \emph{{Metric characterization of Berwald
  spaces of non-positive flag curvature}}, Journal of Geometry and Physics
  \textbf{56} (2006), no.~8, 1257--1270.

\bibitem[KL97]{Kleiner1997}
B.~Kleiner and B.~Leeb, \emph{{Rigidity of quasi-isometries for symmetric
  spaces and Euclidean buildings}}, Publications math{\'{e}}matiques de
  l'IH{\'{E}}S \textbf{86} (1997), no.~1, 115--197.

\bibitem[KOT12]{Kondo2012}
K.~Kondo, S.~Ohta, and M.~Tanaka, \emph{{A Toponogov type triangle comparison
  theorem in Finsler geometry}}, arxiv:1205.3913 (2012).

\bibitem[KS58]{Kelly1958}
P.~Kelly and E.~Straus, \emph{{Curvature in Hilbert geometries}}, Pacific
  Journal of Mathematics (1958).

\bibitem[Kuw13]{Kuwae2013}
K.~Kuwae, \emph{{Jensen's inequality on convex spaces}}, Calculus of Variations
  and Partial Differential Equations \textbf{49} (2013), no.~3-4, 1359--1378.

\bibitem[Lak14]{Lakzian2014}
S.~Lakzian, \emph{{On Closed Geodesics in Non-negatively Curved Finsler
  Structures with Large Volume Growth}}, arxiv:1408.0214 (2014).

\bibitem[{Le }11]{LeDonne2011}
E.~{Le Donne}, \emph{{Metric Spaces with Unique Tangents}}, Annales Academiae
  Scientiarum Fennicae \textbf{36} (2011), no.~2, 683--694.

\bibitem[LV09]{LV2009}
J.~Lott and C.~Villani, \emph{{Ricci curvature for metric-measure spaces via
  optimal transport}}, Ann. of Math. (2) \textbf{169} (2009), no.~3, 903--991.

\bibitem[Lyt06]{Lytchak2006}
A.~Lytchak, \emph{{Open map theorem in metric spaces}}, St. Petersburg
  Mathematical Journal \textbf{17} (2006), no.~3, 477--491.

\bibitem[MR05]{Monti2005}
R.~Monti and M.~Rickly, \emph{{Geodetically Convex Sets in the Heisenberg
  Group}}, Journal of Convex Analysis \textbf{12} (2005), no.~1, 187--196.

\bibitem[Oht07a]{Ohta2007}
S.~Ohta, \emph{{Convexities of metric spaces}}, Geometriae Dedicata
  \textbf{125} (2007), no.~1, 225--250.

\bibitem[Oht07b]{Ohta2007a}
\bysame, \emph{{On the measure contraction property of metric measure spaces}},
  Commentarii Mathematici Helvetici \textbf{82} (2007), no.~4, 805--828.

\bibitem[Oht07c]{Ohta2007p}
\bysame, \emph{Products, cones, and suspensions of spaces with the measure
  contraction property}, Journal of the London Mathematical Society \textbf{76}
  (2007), no.~1, 225--236.

\bibitem[Oht08]{Ohta2008}
\bysame, \emph{{Uniform convexity and smoothness, and their applications in
  Finsler geometry}}, Math. Ann. \textbf{343} (2008), no.~3, 669--699.

\bibitem[Oht09]{Ohta2009}
\bysame, \emph{{Finsler interpolation inequalities}}, Calc. Var. Partial
  Differential Equations \textbf{36} (2009), no.~2, 211--249.

\bibitem[Oht13]{Ohta2013}
\bysame, \emph{{Weighted Ricci curvature estimates for Hilbert and Funk
  geometries}}, Pacific Journal of Mathematics \textbf{265} (2013), no.~1,
  185--197.

\bibitem[OS12]{OS2012}
S.~Ohta and K.-Th. Sturm, \emph{{Non-Contraction of Heat Flow on Minkowski
  Spaces}}, Arch. Ration. Mech. Anal. \textbf{204} (2012), no.~3, 917--944.

\bibitem[Ots97]{Otsu1997}
Y.~Otsu, \emph{{Differential geometric aspects of Alexandrov spaces}},
  Comparison Geometry, Math. Sci. Res. Inst. Publ, 1997, pp.~135--148.

\bibitem[Ped52]{Pedersen1952}
F.~P. Pedersen, \emph{{On spaces with negative curvature}}, Mater. Tidsskr. B
  (1952).

\bibitem[Pet10]{Petrunin2010}
A.~Petrunin, \emph{{Alexandrov meets Lott--Villani--Sturm}}, arxiv:1003.5948
  (2010).

\bibitem[Pet13]{Petrunin2013}
\bysame, \emph{{Alexandrov space with zero reach}},
  http://mathoverflow.net/questions/123922/does-convex-set-in-alexandrov-space-has-positive-reach
  (2013).

\bibitem[PP94]{Perelman1994}
G.~Perelman and A.~Petrunin, \emph{{Quasigeodesics and gradient curves in
  Alexandrov spaces}}, preprint (1994).

\bibitem[Stu06]{Sturm2006a}
K.-Th. Sturm, \emph{{On the geometry of metric measure spaces}}, Acta Math.
  \textbf{196} (2006), no.~1, 65--131.

\bibitem[Sza81]{Szabo1981}
Z.~I. Szab{\'{o}}, \emph{{Positive definite Berwald spaces}}, Tensor N. S.
  \textbf{35} (1981), 25--39.

\bibitem[Sza06]{Szabo2006}
\bysame, \emph{{Berwald metrics constructed by Chevalley's polynomials}},
  arxiv:0601522 (2006).

\bibitem[Yam12]{Yamaguchi2012}
T.~Yamaguchi, \emph{{Collapsing 4-manifolds under a lower curvature bound}},
  arxiv:1205.0323 (2012).

\end{thebibliography}

\end{document}